\documentclass{article}

\usepackage{amsmath}
\usepackage{amsthm}
\usepackage[margin=1.5in]{geometry}

\usepackage{times}
\usepackage{bm}
\usepackage{natbib}

\usepackage{amsfonts}

\usepackage[plain,noend]{algorithm2e}
\usepackage{graphicx,caption}

\makeatletter
\renewcommand{\algocf@captiontext}[2]{#1\algocf@typo. \AlCapFnt{}#2} 
\def\@algocf@capt@plain{top}
\renewcommand{\algocf@makecaption}[2]{%
  \addtolength{\hsize}{\algomargin}%
  \sbox\@tempboxa{\algocf@captiontext{#1}{#2}}%
  \ifdim\wd\@tempboxa >\hsize
    \hskip .5\algomargin%
    \parbox[t]{\hsize}{\algocf@captiontext{#1}{#2}}
  \else%
    \global\@minipagefalse%
    \hbox to\hsize{\box\@tempboxa}
  \fi%
  \addtolength{\hsize}{-\algomargin}%
}
\makeatother

\newcommand{\be}{ {\bf e} }
\def\E{\mathbb{E}}
\newcommand{\mR}{\mathbb{R}}

\newtheorem{proposition}{Proposition}
 \newtheorem{corollary}{Corollary}
 \newtheorem{lemma}{Lemma}
 \newtheorem{definition}{Definition}

 \newtheorem{remark}{Remark}

\begin{document}



\markboth{I. M. Johnstone and B. Nadler}{Distribution of Roy's Test}

\title{Roy's Largest Root Test Under Rank-One Alternatives}

\author{I.M. Johnstone and B. Nadler}

\maketitle

\begin{abstract}
 Roy's largest root is a common test statistic in 
multivariate analysis, statistical
signal processing and allied fields.
Despite its ubiquity, provision of
  accurate and tractable approximations to its distribution under the
  alternative has been a longstanding open problem.
Assuming Gaussian observations and a
  rank one alternative, or {\em concentrated
    non-centrality}, we derive simple yet accurate approximations 
for the most common low-dimensional settings.  These include signal detection in
noise, multiple response regression, 
  multivariate analysis of variance and canonical correlation
  analysis.  
A  small noise perturbation approach, perhaps underused in statistics, 
leads to simple combinations of standard univariate
  distributions, such as central and non-central $\chi^2$ and $F$.
  Our results allow approximate power and sample size calculations for
  Roy's test for  rank one effects, which is precisely where it is most powerful.
\end{abstract}



\section{An example}

Before providing fuller context, we begin with multiple response linear
regression, an important particular case. 
Consider a linear model with $n$ observations on an $m$ variate
response 
\begin{equation}
  \label{eq:regmodel}
   Y = X B + Z, 
\end{equation}
where $Y$ is $n \times m$ and the known design matrix
$X$ is $n \times p$, so that the unknown coefficient matrix
$B$ is $p \times m$. Assume that $X$ has full
rank $p$.
The Gaussian noise matrix $Z$ is assumed to have independent rows,
each with mean zero and covariance $\Sigma$, thus
$Z \sim N(0, I_n \otimes \Sigma)$.

A common null hypothesis is $C B = 0$, for some  
\(n_H\times p\) contrast matrix $C$ of full rank $n_H \leq p$.
This is used, for example, to test (differences among) subsets of
coefficients \citep{stat1999sas,johnson_wichern}. 
Generalizing the univariate $F$ test, it is traditional to form
hypothesis and error sums of squares and cross products
matrices, which under our Gaussian assumptions have independent
Wishart distributions:
\begin{align*}
  H  = Y ^T P_H Y  \sim W_m( n_H, \Sigma, \Omega), \quad\quad 
  E  = Y ^T P_E Y  \sim W_m( n_E, \Sigma).
\end{align*}
Full definitions and formulas are postponed to Section \ref{s:setup} and
\eqref{eq:reg-details}
below; for now we note that $P_E$ is
orthogonal projection of rank $n_E = n-p$
onto the error subspace, $P_H$ is orthogonal
projection of rank $n_H$ 
  on the hypothesis subspace for $C B$,
and $\Omega$ is the non-centrality matrix corresponding to
the regression mean $\E Y = XB$.

Classical tests use the eigenvalues of the $F$-like
matrix $E^{-1} H$:
the ubiquitous four are Wilks' $U$, 
the Bartlett-Nanda-Pillai $V$,
the Lawley-Hotelling $W$ and, 
our focus here, Roy's
largest root $R = \ell_1(E^{-1} H) $, based on the largest
eigenvalue.
The first three have long had adequate approximations.
Our new approximation for $R$, valid for the case of rank one non-centrality
matrix $\Omega$, employs a linear combination of two
independent  $F$ distributions, one of which is noncentral.
The following result, for multivariate linear regression, is the fourth in our sequence of conclusions below.

\setcounter{proposition}{3}
\begin{proposition}
\label{prop:MANOVA}
  Suppose that $H\sim W_m( n_H, \Sigma, \Omega)$ and
$E \sim W_m( n_E, \Sigma)$ are independent Wishart matrices with
$m > 1$
 and $\nu = n_E - m > 1$.
Assume that the non-centrality matrix has rank one, 
$\Omega=\omega \mu\mu^T$, where $\|\mu\|=1$. 
If $m, n_H$ and $n_E$ remain fixed and $\omega \rightarrow
\infty$, then
\begin{equation}
  \label{eq:ell1-approx}
  \ell_1(E^{-1}H)
    \approx c_1 F_{a_1,b_1}(\omega) + c_2 F_{a_2,b_2} + c_3,
\end{equation}
where the $F$-variates are independent, and 
the parameters \(a_{i},b_i,c_i\) are given by
\begin{alignat}{2}
\label{eq:ab-pars}
   a_1 & = n_H, \qquad &
   b_1 & = \nu +1, \qquad
   a_2 = m -1, \qquad
   b_2 = \nu +2, \\
\label{eq:c-pars}
    c_1 & = a_1/b_1, \qquad &
  c_2 & =  a_2/ b_2,  \qquad
  c_3 = a_2/(\nu(\nu-1)).
\end{alignat}
\end{proposition}

The approximation \eqref{eq:ell1-approx}
is easy to implement in software such as \textsc{Matlab} or R
- a single numerical integration on top of standard built in functions is all
that is required.
Figure \ref{fig:hist_L1} (right panel) shows the approximation in action with
$m = 5, n_H = 4, 
n_E = 35$ and with non-centrality $\omega = 40$.
The approximated density
matches quite closely the empirical one and both are far from 
the nominal limiting Gaussian density.
The error terms in the approximation $\approx$ in Eq. (\ref{eq:ell1-approx}) are discussed in
Section \ref{s:Distribution_Roys_Test} and the Supplement.

This and other results of this paper can aid power analyses and sample size design in exactly
those settings in which Roy's test may be most appropriate, namely when
the relevant alternatives are thought to be predominantly of rank one.
Table \ref{tab:simulation} gives a small illustrative example in the multivariate
anova special case of
Proposition \ref{prop:MANOVA}, reviewed in Section \ref{s:setup}. 
The $k= 1,  \ldots, p$ group means are assumed under the alternative
to vary as multiples  
$\mu_k = k \tau \mu_0$ of a fixed vector $ \mu_0=(1,1,1,-1,-1,-1)\in\mathbb{R}^6$,
with scale factor $\tau$.
The noise covariance matrix is \(\Sigma=(1-\rho)I+\rho{\bf 1}{\bf 1}^T\), and there are 20 samples in each group. 
This setting leads to a rank one non-centrality matrix for which, assuming Gaussian observations, Proposition 4 applies.  
\begin{table}[h]
\centering
{Power Comparison of Pillai's and Roy's tests\\}{%
\begin{tabular}[h]{lrrr}
  Power & $\tau = 0.09$  & $\tau = 0.11$  & $\tau = 0.13$ \\[5pt] 
Pillai trace $\rho = 0$ & 0.25  & 0.47 & 0.70 \\
Largest root $\rho = 0$ & 0.35  & 0.64 & 0.87 \\
Approximation from proposition 4 & 0.30  & 0.60 & 0.86 \\[5pt]
Pillai trace $\rho = 0.3$ & 0.44  & 0.72  & 0.91 \\
Largest root $\rho = 0.3$ & 0.60  & 0.88  & 0.98 \\
Approximation from proposition 4 & 0.58  & 0.87 & 0.98 
\end{tabular}}
\caption{
$N = 1,000,000$ simulations, $p = 6, m = 6$, $SE \leq .00015$ 
}
   \label{tab:simulation}
        \label{tab:example}
\end{table}

Table \ref{tab:example} compares the power of multivariate tests at three
signal strengths and two correlation models at level \(1\%\). 
The 
Pillai trace $V$ is chosen as representative of the three tests that use
all the roots.  
Especially in the larger two signal
settings, Roy's largest root test makes the difference between a
plausible experiment and an underpowered study. For a detailed real manova example in which Roy's test is argued to be most appropriate, see  \cite[Study C]{hand1987multivariate}.

Simulations like this can suggest the magnitude of improvement possible in
selected cases, but the essence of power and sample size analysis is
the comparison of a range of scenarios thought to encompass the likely
experimental setting. For this, relatively simple approximate formulas 
such as \eqref{eq:ell1-approx} are
invaluable.


\section{Introduction}
\label{sec:introduction}

Hypothesis testing plays an important role in the analysis of
multivariate data.  Classical examples include the
multiple response linear model, principal components and canonical
correlation analysis, as well as others which together form the main
focus of standard
multivariate texts, e.g. \cite{ande03,mkb79,must06}.
They find widespread use in signal processing, social sciences and
many other domains.

Under multivariate Gaussian assumptions, in all these cases the
associated hypothesis tests can be formulated in terms of
either one or two independent Wishart matrices.
These are conventionally denoted $H$, for hypothesis, and
$E$, for error, depending on whether the covariance matrix $\Sigma$ is
known or unknown--in the latter case $E$ serves to estimate $\Sigma$.

\citet{James64} provided a remarkable five-way classification of the
distribution theory associated with these problems.
Elements of the classification are indicated in Table
\ref{t:James5cases},
along with some representative applications.
Departure from the null hypothesis is captured by a
matrix $\Omega$, 
so that the testing problem might be 
$  \mathcal H_0:\,\Omega=0$  vs. 
$\mathcal H_1:\, \Omega \neq 0$.
Depending on the particular application, the matrix $\Omega$ captures
the difference in group means, or the number of signals or canonical
correlations and their strengths.  In the absence of detailed
knowledge about the structure of $\Omega$ under $\mathcal H_1$,
group invariance 
arguments show that generic tests depend on the
eigenvalues of either $\Sigma^{-1} H$ or $E^{-1} H$,
 e.g., \citet{Muirhead_book}.

\begin{table}[htbp]%
  \def~{\hphantom{5}}
\begin{center}
Classification of multivariate problems\\
{%
\begin{tabular}{cccll}
Case &  Multivariate     & Distr. for & dimension &    \ \ \ \ Testing
Problem,        \\     &      Distribution     & dim. $m=1$     & $m >
1$   &  \ \ \ \ \ \ \ \ Application \\[5pt]
1 &   $\overset{\,}{\,_0F_0}$ &  $\chi^2$       &               $H\sim
W_m(n_H,\Sigma+\Omega)$ &   Signal detection in noise,  \\  &        &
&   $\Sigma$\ \  known                    &   known covariance matrix\\[3pt]
2 & $\,_0F_1$ &  non-central     &   $H\sim W_m(n_H,\Sigma,\Omega)$
&   Equality of group means, \\
        &                                  &   $\chi^2$        &
        $\Sigma$\ \ known,                         &   known covariance
        matrix \\[3pt]
3 & $\overset{\,}{\,_1F_0}$ &  $F$       &   $H\sim
W_m(n_H,\Sigma+\Omega)$ &   Signal detection in noise, \\
  &         &                          &   $E\sim W_m(n_E,\Sigma)$
  &   estimated covariance  \\[3pt]
4 & $\,_1F_1$ &  non-central         &   $H\sim W_m(n_H,\Sigma,\Omega)$
&   Equality of group means, \\
 &                                              &  $F$
 &   $E\sim W_m(n_E,\Sigma) $           &   estimated covariance  \\[3pt]
5 & $\overset{\,}{\,_2F_1}$ &  Correlation coeff.  &   $H\sim
W_p(q,\Sigma,\Omega)$ &   Canonical Correlation \\
  &                        &   $r^2/(1-r^2)$          &  $ E\sim
  W_p(n-q,\Sigma)$,  &   Analysis between two  \\
  &                        &   $t$-distribution           &  $\Omega$
  itself random &      groups of sizes $p\leq q$.                \\[3pt]
\end{tabular}}
\end{center}
\caption{
James' classification of eigenvalue distributions 
  was based on hypergeometric
  functions $_aF_b$ of matrix argument; their univariate analogs are
  shown in column 3.  Column 4 details the corresponding Wishart
  assumptions for the sum of squares and cross products matrices; the
  final column gives a non-exhaustive list of sample applications.  
}
\label{t:James5cases}
\end{table}

The most common tests  fall into two 
categories.
The first consist of `linear' statistics, 
which depend on \textit{all} the eigenvalues, and are expressible in the form
$\sum_i f(\ell_i)$ for some univariate function $f$.
This class includes the test statistics $U, V, W$ already mentioned, 
e.g. \citet{Muirhead_book,ande03}.

The second category involves functions of the \textit{extreme}
eigenvalues--the first few largest and smallest.
Here we focus on the largest root statistic, based on $\ell_1$, which
arises systematically in multivariate analysis as the
union-intersection test \citep{Roy_57}. 
To summarize extensive simulations by \citet{Schatzoff} and
\citet{Olson_74}, Roy's test is most powerful among the
common tests when the alternative is of rank one, i.e. ``concentrated
noncentrality''.
For fixed dimension, \citet{krna09} showed asymptotic (in sample size)
optimality of Roy's test against rank one alternatives.

We briefly contrast the state of knowledge regarding approximate
distributions, both null and alternate, for the two categories of test
statistics.
For the linear statistics,  approximations using an $F$ distribution
are traditional and widely available in software: central $F$ under the null \citep{stat1999sas} and non-central
\(F\) under the alternative \citep{Muller_84,Muller_92,OBrien_Shieh}.
Saddlepoint approximations \citep{buwo05,Butler_Paige} are also
available.

For Roy's largest root test, the situation is less complete.
In principle, the distribution of the largest eigenvalue has an exact
representation in
terms of a hypergeometric function of matrix argument. 
Despite recent advances in the numerical evaluation of these special
functions \citep{koed06}, 
unless dimension and sample size are small, say $< 15$,
these formulas are challenging
to evaluate numerically. 
Under the \textit{null}, 
\cite{Butler_Paige} 
as well as \cite{Chiani12,Chiani14}
derived fast and accurate methods for numerical 
evaluation of the null distribution of Roy's test.
Still under the null, instead of exact calculations, 
simple asymptotic approximations to Roy's test
can be derived from random matrix theory in the
high dimensional setting:
\citet{elka04p,john08p,Johnstone_MANOVA,ma11}.

In contrast, under the \textit{alternative},
derivation of a simple  approximation to the
distribution of $\ell_1$
 has remained a longstanding
problem in multivariate analysis.  To date, for dimension $m>2$, no acceptable method has been developed for transforming Roy's
largest root test statistic to an $F$ or $\chi^2$ statistic, and no
straightforward method exists for computing powers for Roy's statistic
itself, as noted in 
\citet[p. 332]{ande03}, \citet{Muller_92,OBrien_Shieh}.

\smallskip
\textit{Contributions of the paper.} \ 
We develop simple and quite accurate approximations
for the distribution of $\ell_1$ 
for the classic problems of multivariate analysis, Table
\ref{t:James5cases}, under a rank-one 
alternative.
Under this {\em concentrated non-centrality} alternative,
the noncentrality matrix has the form
$\Omega = \omega v v^T, \omega > 0$,
where $ v \in\mathbb{R}^p$ is an arbitrary and unknown unit norm
vector.
This setting, in which $\ell_1$ is approximately the most powerful
test, may be viewed as a specific
form of sparsity, indicating that the effect under study can be
described by relatively few parameters.

Our approach keeps $(m, n_H, n_E)$ \textit{fixed}.
We study the limit of large
non-centrality parameter, or equivalently small noise.
While small noise perturbation is a classical method in 
applied mathematics and mathematical physics, 
it has apparently seen less use in statistics. Some exceptions, mostly focused on other multivariate problems, include \citet{kada70,kada71},
\cite{Anderson_77,Schott_86,naco05} and \cite{Nadler_AOS}.

Our small-noise analysis uses tools from matrix
perturbation theory and yields an approximate {\em stochastic
  representation} for $\ell_1$.
In concert with standard Wishart results, we
deduce its approximate distribution for the five cases
of Table \ref{t:James5cases}
in Propositions
\ref{prop:h1} through \ref{prop:CCA}.
The expressions obtained can be readily
evaluated numerically, typically via a single
integration\footnote{ 
\textsc{Matlab} code for the resulting distributions and
  their power will be made available at the author website,
http://www.wisdom.weizmann.ac.il/$\sim$nadler.}.

Finally, it is not the purpose of this paper to argue for a general and
unexamined use of Roy's test. It is well established that there is
no uniformly best test, and in particular settings issues of
robustness to non-normality already studied e.g. by Olson (1974) may
be important.  Instead, when there \textit{is} interest in the performance of the
largest root in the rank one Gaussian cases where it should shine, we
provide approximations that have long been lacking.


\section{Definitions and Two Applications}
\label{s:setup}

We present two
applications, one from multivariate statistics and the other
from signal processing, that illustrate Settings 1-4 of Table
\ref{t:James5cases}.
Following \citet[p. 441]{Muirhead_book},
we recall that if
$z_i \stackrel{\text{ind}}{\sim} N_m(\mu_i, \Sigma)$ for
$i=1, \ldots, n$ with $Z^T = [z_1, \cdots,
z_n]$
and $M^T = [ \mu_1, \cdots, \mu_n]$ then
the $m \times m$ matrix
$A = Z^T Z$ is said to have the noncentral Wishart
distribution $W_m(n, \Sigma, \Omega)$ with $n$ degrees of freedom,
covariance matrix $\Sigma$ and
noncentrality matrix $\Omega = \Sigma^{-1} M^T M$, which may also be written in the symmetric form
$\Sigma^{-1/2}M^T M\Sigma^{-1/2}$.
When $\Omega = 0$, the distribution is a central Wishart, $W_m(n,
\Sigma)$.

\medskip
\textit{Signal Detection in Noise.} \ 
Consider a measurement system consisting of $m$ sensors (antennas,
microphones, etc).  In the signal processing literature, see for
example \cite{Kay}, a standard model for the observed samples in the
presence of a {\em single} signal is
\begin{equation}
x = \sqrt{\rho_s} u h + \sigma \xi
        \label{eq:x_SP}
\end{equation}
where $h$ is an unknown $m$-dimensional
vector, assumed fixed
during the measurement time window, $u$
is a random variable distributed ${\cal N}(0,1)$,
$\rho_s$ is the signal strength, $\sigma$ is the
noise level and
$\xi$ is a random noise vector, independent of $u$, that is
multivariate Gaussian  ${\cal N}_m (0,
\Sigma)$.

In this paper, for the sake of simplicity, we assume real valued
signals and noise. The complex-valued case
can be handled in a similar manner \citep{dharmawansa2014roy}.
Thus, let $x_i\in\mathbb{R}^m$ denote
$n_H$ i.i.d. observations from
Eq. (\ref{eq:x_SP}), and let $n_H^{-1}H$ denote their sample covariance matrix,
\begin{equation}
H = \sum_{i=1}^{n_H} x_ix_i^T \sim W_m(n_H, \sigma^2 \Sigma + \Omega),
        \label{eq:def_H_SP}
\end{equation}
where $\Omega = \rho_s h h^T$ has rank one.
A fundamental task in signal
processing is to test
$\mathcal{H}_0: \rho_s = 0$, no signal present, versus
$\mathcal{H}_1: \rho_s > 0.$
If the covariance matrix ${ \Sigma}$ is known, setting 1 in
Table \ref{t:James5cases}, the observed data can be {\em whitened} by
the transformation ${ \Sigma}^{-1/2} x_i$.
Standard detection schemes then depend on the eigenvalues of 
$\Sigma^{-1} H$, \citet{waka85,krna09}.

A second important case, Setting 3, assumes that
the noise
covariance matrix ${ \Sigma}$ is {\em arbitrary and unknown},
but we have additional ``noise-only'' observations
$z_j \sim \mathcal{N}(0,\Sigma)$ for $j = 1,
\ldots, n_E$.
It is then traditional to estimate the noise covariance by $n_{E}^{-1} E$,
where
\begin{equation}
E = \sum_{i=1}^{n_E} z_i z_i^T \sim W_m( n_E, \Sigma),
        \label{eq:def_E_SP}
\end{equation}
and devise detection schemes using the eigenvalues of \(E^{-1}H\).
Some representative papers on signal detection in this setting,
and the more general scenario with several sources, include
\cite{Zhao_86}, \cite{Zhu_91,Stoica_Cedervall} and
\cite{Rao_Silverstein}.

\medskip
\textit{Multivariate Analysis of Variance.} \
The comparison of means from $p$ groups is a common and simple special
case of the regression model \eqref{eq:regmodel}, and suffices to
introduce Settings 2 and 4 of Table \ref{t:James5cases}.
Let $I_k$ index observations in the $k$-th group, $k = 1, \ldots, p$
and assume a model
\begin{displaymath}
  y_i = \mu_k + \xi_i,  \qquad \qquad i \in I_k.
\end{displaymath}
Here $\xi_i \stackrel{\text{ind}}{\sim}
\mathcal{N}_m (0,\Sigma)$ with the error covariance $\Sigma$  assumed to be the same for all groups, and the indices
$\{1, \ldots, n \} = I_1 \cup \cdots \cup I_p$, with 
$n_k = |I_k|$ and $n = n_1 + \cdots + n_p$.
We test the equality of group means:
$\mathcal{H}_0: \mu_1 = \ldots = \mu_p$ versus the
alternative $\mathcal{H}_1$ that the $\mu_k$ are not all
equal.
A known \(\Sigma\) leads to Setting 2. When it is unknown we obtain Setting 4.

The standard approach is then to form \textit{between} and \textit{within}   group covariance matrices:
\begin{equation*}
   H = \sum_{k=1}^p n_k (\bar{y}_k - \bar{y})
               (\bar{y}_k - \bar{y})^T, \qquad
   E  = \sum_k \sum_{i \in I_k}
          (y_i - \bar{y}_k) (y_i - \bar{y}_k)^T
\end{equation*}
where 
$\bar{y}_k$ and $\bar{y}$ are the group and overall sample
means respectively.
The independent Wishart distributions of $H$ and $E$ appear in rows 2
and 4 of Table \ref{t:James5cases}. The degrees of freedom are given by $n_H = p-1$ and $n_E = n-p$.
If $\bar{\mu} = n^{-1} \sum n_k \bar{\mu}_k$ is
the overall population mean, 
then the noncentrality matrix is
$\Omega = \Sigma^{-1} \sum_1^p n_k (\mu_k -
\bar{\mu})  (\mu_k - \bar{\mu})^T$.

A rank one non-centrality matrix is obtained if we assume 
that under the alternative, the means of the
different groups are all proportional to the {\em same }
unknown vector ${\mu}_0$, with each multiplied
by a group dependent strength parameter.
That is, ${\mu}_k = s_k {\mu}_0$.
This yields a rank one non-centrality matrix $\Omega =
\omega {\Sigma}^{-1}{\mu}_0 {\mu}_0^T$, 
where $\bar{s} = n^{-1} \sum_k n_k s_k$,
and $\omega = \sum_{k=1}^p n_k (s_k - \bar{s})^2$.

\section{On the Distribution of the Largest Root Test}
\label{s:Distribution_Roys_Test}

Let $\ell_1$ be the largest eigenvalue of either $\Sigma^{-1}H$ or
$E^{-1}H$, depending on the specific setting. Roy's test 
rejects the null if $\ell_1 > t(\alpha)$
where $t(\alpha)$ is the threshold corresponding to a false alarm
or type I error rate  $\alpha$.
The probability of detection, or power of Roy's test is defined as
\begin{equation}
P_D = P_{D,\Omega}
= \Pr\left[\ell_1> t(\alpha) \,|\,{\cal H}_1\right].
        \label{eq:P_D}
\end{equation}

Under the null hypothesis, $\Omega = 0$, 
and in all null cases, $H$ has a central $W_m(n_H, \Sigma)$
distribution, and the distinction between settings (1,2) and (3,4) is
simply the presence or absence of $E \sim W_m(n_E, \Sigma)$.
An exact or approximate threshold $\ell_1(\alpha)$ may be
found by the methods referenced in Section \ref{sec:introduction}.
The focus of this
paper is on the power $P_D$ under rank-one alternatives.
To this end, we present simple approximate expressions
for the distribution of Roy's largest root statistic $\ell_1$ for all five settings
described in Table \ref{t:James5cases}, under a rank-one alternative.
Detailed proofs appear in the appendix and Supplement. 

We begin with Cases 1 and 2, where the matrix $\Sigma$ is assumed to be known. Then,
without loss
of generality we instead study the largest eigenvalue of $\Sigma^{-1/2}H\Sigma^{-1/2}$, distributed as \(W_{m}(n_H,\sigma^2 I+\lambda_Hvv^T)\), for some suitable vector $v\in\mathbb{R}^m$.

\setcounter{proposition}{0}
\begin{proposition} \label{prop:h1}
Let $H \sim W_m(n_H, \sigma^2 I + \lambda_H {v} {v}^T)$
with $\| {v} \| = 1,\lambda_H>0$ and let $\ell_1$ be its largest eigenvalue.
Then, with $(m, n_H, \lambda_H)$ fixed, as $\sigma \rightarrow 0$
\begin{equation}
  \label{eq:h1def}
  \ell_1 = (\lambda_H + \sigma^2) \chi_{n_H}^2
        + \chi_{m-1}^2 \sigma^2
        + \frac{\chi_{m-1}^2 \chi_{n_H-1}^2}{(\lambda_H +
          \sigma^2) \chi_{n_H}^2} \sigma^4 + o_p(\sigma^4),
\end{equation}
where the three variates $\chi_{n_H}^2, \chi_{m-1}^2$ and
$\chi_{n_H-1}^2$ are independent.
\end{proposition}

\begin{proposition} \label{prop:h1_M}
Now let $H \sim W_m(n_H, \sigma^2 I, (\omega/\sigma^2) {v} {v}^T)$
with $\| {v} \| = 1,\omega>0$ and let $\ell_1$ be its largest eigenvalue.
Then, with $(m, n_H, \omega)$ fixed, as $\sigma \rightarrow 0$
\begin{equation}
  \label{eq:h1defMANOVA}
  \ell_1 = \sigma^2 \chi_{n_H}^2 (\omega/\sigma^2)
        + \chi_{m-1}^2 \sigma^2
        + \frac{\chi_{m-1}^2 \chi_{n_H-1}^2}{\sigma^2 \chi_{n_H}^2
          (\omega/\sigma^2)} \sigma^4 + o_p(\sigma^4),
\end{equation}
where the three variates $\chi_{n_H}^2, \chi_{m-1}^2$ and 
$\chi_{n_H-1}^2(\omega/\sigma^2)$ are independent.
\end{proposition}

\begin{remark}
  If $\sigma^2$ is  held fixed (along with $m, n_H$) in the above
  propositions and instead we suppose $\lambda_H$ (resp. $\omega$)
  $\to \infty$, then
  the same expansions hold, now with error terms $o_p(1/\lambda_H)$
  (resp. $o_p(1/\omega)$).
  Lest it be thought unrealistic to base approximations on large
  $\lambda_H$, small $\sigma$, or in Case 5  $\rho$ near 1, we remark
  that in standard simulation
  situations they lead to levels of power conventionally regarded as desirable, see Section \ref{s:simulations}; indeed in these cases,
  weaker signals would not be 
  acceptably detectable.
\end{remark}

Approximations to the moments of $\ell_1$
follow directly. From \eqref{eq:h1def},
independence of the chi-square variates and
standard moments formulas, we have
\begin{equation}
  \label{eq:case1-mean}
  \E \ell_1 \approx
    n_H \lambda_H + (m-1+n_H) \sigma^2 + \frac{(m-1)(n_H-1)}{(\lambda_H
      + \sigma^2)(n_H-2)} \sigma^4.
\end{equation}
In case 2, we simply replace $n_H \lambda_H$ by $\omega$ in the first
term, and the denominator of the third term by $\omega +
\sigma^2(n_H-2)$. 

To compare the variances $\text{var} (\ell_1)$ in
cases 1 and 2, it is natural to set
$\omega = \lambda_H n_H$, so
that the means match to the leading two orders.
If  $\sigma = 1$ and $\lambda_H = \omega/n_H$ is large, then
\begin{equation}
  \label{eq:cases12-var}
  \text{var} (\ell_1) =
  \begin{cases}
    2 n_H \lambda_H^2  + 4 n_H \lambda_H + 2(m-1+n_H) + o(1) & \text{case 1} \\
    4 n_H \lambda_H + 2(m-1+n_H) + o(1) 
 & \text{case 2}
  \end{cases}
\end{equation}
Thus, for $\lambda_H \gg 1$,  the fluctuations of $\ell_1$ in case 2
are significantly smaller.
While beyond the scope of this paper, this result has implications for
the detection power of Gaussian signals versus those of constant modulus.


Next, we consider the two matrix case, where $\Sigma$ is unknown and
estimated from data.
Consider first the signal detection setting. 

\begin{proposition}\label{prop:SP}
  Let $H\sim W_m( n_H, \Sigma + \lambda_H v v^T)$ and
$E \sim W_m( n_E, \Sigma)$ be independent Wishart matrices, 
with
$m > 1$ and 
$v^T\Sigma^{-1}v = 1$.
If $m, n_H$ and $n_E$ are fixed and $\lambda_H \rightarrow
\infty$, then
\begin{equation}
  \label{eq:L_1_SP}
  \ell_1(E^{-1} H)
    \approx c_1 (\lambda_H +1) F_{a_1,b_1} + c_2 F_{a_2,b_2} + c_3.
\end{equation}
where the $F$-variates are independent, and with $\nu = n_E - m > 1$,
the parameters \(a_{i},b_i,c_i\) are given by
\eqref{eq:ab-pars} and 
\eqref{eq:c-pars}.
\end{proposition}

Proposition 4, the corresponding result for multiple response regression, 
appears at the start of the paper. The parameters used there
are related to those in
 model \eqref{eq:regmodel}
as follows: with $M = X B$,
\begin{equation}
  \label{eq:reg-details}
  \begin{split}
  P_E & = I - X (X^T X)^{-1} X^T,
 \\
  P_H & = X (X^T X)^{-1}
  C^T [ C (X^T X)^{-1}  C^T]^{-1}
          C (X^T X)^{-1} X^T, \\
   n_E &  = \text{rank}(P_E) = n - p, \qquad
   n_H = \text{rank}(P_H) \\
  \Omega & = \Sigma^{-1} M^T P_H M
       = \Sigma^{-1} B^T C^T
            [C (X^T X)^{-1}
            C^T]^{-1}
              C B,
  \end{split}
\end{equation}
see, e.g. in part, \citet[Sec 6.3.1]{mkb79}.

In the \(n_E\to\infty\) limit, the two $F$-variates
in \eqref{eq:ell1-approx} and 
\eqref{eq:L_1_SP}
converge to $\chi^2$ variates and 
we recover the first two terms in
the approximations of
Propositions \ref{prop:h1} and \ref{prop:h1_M}
(with $\sigma^2 = 1$ held fixed).

We turn briefly to the approximation errors in \eqref{eq:ell1-approx}
and \eqref{eq:L_1_SP}.
When $m=1$, we have $c_2 = c_3 = 0$ and 
the first term gives the \textit{exact} distribution of $H/E$ for
both Propositions \ref{prop:SP} and \ref{prop:MANOVA}. 
For $m > 1$, in Case 4, we note that to leading order 
$\ell_1 = O_p((\omega + n_H)/n_E)$,
whereas 
the errors arise from ignoring terms
$O_p(\omega^{-1/2})$
and higher in
an eigenvalue expansion, and by replacing stochastic terms of order
$O_p( (\omega + n_H)^{1/2} m^{1/2} n_E^{-3/2})$
by their expectations.
The corresponding statements apply for Case 3 if we replace
$\omega + n_H$ by $\lambda_H n_H$ and $\omega^{-1/2}$ by 
$\lambda_H^{-1/2}$.
Detailed discussion appears in the Supplement at
\eqref{eq:Var-Ti},~\eqref{eq:Var-Ti1} and subsection
\ref{sec:analysis-error-terms}.

We now turn to expressions for $\E \ell_1$ and $\text{var}(\ell_1)$ 
in Cases 3 and 4, analogous to \eqref{eq:case1-mean}
and \eqref{eq:cases12-var}.

\begin{corollary}
  In Case 4,
\begin{equation}
  \label{eq:ell1mean}
  \mathbb{E} \, \ell_1( E^{-1} H)
    \approx \frac{\omega+n_H}{n_E - m - 1}
     + \frac{m - 1}{n_E-m}.
\end{equation}
\begin{equation}
  \label{eq:ell1var}
  \text{var} \, \ell_1( E^{-1} H)
   \approx \frac{2[\omega^2 + \nu n_H (n_H+2 \omega)
]}{p_3(\nu-1)} + 
     \frac{2(m-1)(n_E-1)}{p_3(\nu)},
\end{equation}
where $p_3(\nu) = \nu^2 (\nu-2)$.
In Case 3, $\omega$ is replaced by $\lambda_H n_H$ and in
\eqref{eq:ell1var}, the term $n_{H} + 2 \omega$ is increased to $n_H(\lambda_H +1)^2$.
\end{corollary}

Let $\hat \Sigma = n_E^{-1} E$ be an unbiased estimator of $\Sigma$. 
Comparison with Propositions \ref{prop:h1} and \ref{prop:h1_M} shows
that 
$\mathbb{E} \, \ell_1( \hat \Sigma^{-1} H)$ 
exceeds
$\mathbb{E} \, \ell_1( \Sigma^{-1} H)$ by a multiplicative
factor close to $n_E/(n_E - m - 1)$, so that the largest eigenvalue of
$n_E E^{-1}H$ is thus typically larger than that of the matrix ${
  \Sigma}^{-1} H$.
Again, the fluctuations of $\ell_1$ in the MANOVA setting are smaller
than for signal detection.

\citet{Rao_Silverstein}  studied the large parameter limiting value 
(but not the distribution) of 
$\ell_1( E^{-1} H)$ as
$m/n_H \rightarrow c_E, m/n_H \to c_H$, also
in non-Gaussian cases.
In this limit, our formula \eqref{eq:ell1mean}
agrees, to leading order terms, with the large $\lambda_H$ limit of their expression (Eq. (23))
.
Hence, our analysis shows that the limits 
for the mean of $\ell_1(E^{-1}H)$ are quite accurate even at
 smallish values of $m,n_E,n_H$.
This is also reflected in our simulations in Section
\ref{s:simulations}.

\subsection{Canonical Correlation Analysis}

Let $\{{x}_i\}_{i=1}^{n+1}$ denote $n+1$
multivariate Gaussian observations on $m=p+q$ variables
with unknown mean $\mathbf{\mu}$ and covariance
matrix $\Sigma$, and let $S$ denote the 
mean-centered sample covariance.
Assume without loss of generality that $p \leq q$ and 
decompose $\Sigma$ and $S$ as
\begin{equation}
\Sigma =
\begin{pmatrix}
\Sigma_{11} & \Sigma_{12} \\
\Sigma_{21} & \Sigma_{22}
\end{pmatrix},
\qquad
S =
\begin{pmatrix}
S_{11} & S_{12} \\
S_{21} & S_{22}
\end{pmatrix}
\label{eq:Sigma_CCA}
\end{equation}
where $\Sigma_{11}$ and $\Sigma_{22}$ are square matrices of sizes $p\times p$
and $q\times q$, respectively.
We might alternatively assume that $\mu = 0$ is known and that we
have $n$ independent observations.
In either case, the parameter $n$ denotes the degrees of freedom of
the Wishart matrix $n S$.

The population and sample canonical correlation coefficients, denoted
$\rho_1,\ldots,\rho_p$ and $r_1,\ldots,r_p$,
are the positive square roots of the
eigenvalues of $\Sigma_{11}^{-1} \Sigma_{12}\Sigma_{22}^{-1}
\Sigma_{21}$ and 
$S_{11}^{-1} S_{12}S_{22}^{-1} S_{21}$.
We study the distribution of the largest sample
canonical correlation, in the presence of a single large population
correlation coefficient, $\rho_1 > 0, \rho_2=\ldots,\rho_p=0$.


To state our final proposition, we need 
a modification of the non-central $F$ distribution that
is related to the squared multiple correlation coefficient.

\begin{definition}
  A random variable $U$ follows a \textit{$\chi^2_n$-weighted non-central $F$
    distribution}, with parameters $a,b,c,n$, written 
  $F_{a,b}^\chi(c,n)$, if it has the form
\begin{equation}
U = \frac{\chi_a^2(Z)/a}{\chi_b^2/b}
        \label{eq:XF_distribution}
\end{equation}
where the non-centrality parameter $Z \sim c\chi^2_n$ is itself a
random variable, and all three chi-squared variates are independent.
\end{definition}

If $c=0$,
the $F^\chi$ distribution of \(U\) reduces to a central $F$.
For $c > 0$, the distribution 
is easily evaluated numerically
via either of the representations
\begin{displaymath}
  P(U \leq u)
      = \int_0^\infty p_n(t) F_{a,b;ct}(u) dt
      = \sum_{k=0}^\infty p_K(k) F_{a+2k,b} (au/(a+2k)).
\end{displaymath}
In the first, $F_{a,b;\omega}$ is the non-central $F$ distribution 
with non-centrality $\omega$
and $p_n$ is the density of $\chi_n^2$ -- this is just the definition.
In the second, $p_K$ is the discrete p.d.f. of a negative binomial
variate with parameters $(n/2,c)$: this is an analog of the more
familiar representation of noncentral $F_{a,b;\omega}$ as a mixture of
$F_{a+ 2k,b}$ with Poisson($\omega/2$) weights.
The equality above
may be verified directly, or from \citet[p. 175ff]{Muirhead_book},
who also gives 
an expression for the $F^\chi$ distribution
  in terms of the Gauss hypergeometric function ${}_2  F_1$. 

\setcounter{proposition}{4}
\begin{proposition} \label{prop:CCA}
Let $\ell_1 = r_1^2/(1-r_1^2)$, where $r_1$ is the largest sample
canonical correlation between two groups of sizes $p\leq q$ computed
from $n+1$ i.i.d. observations, with $\nu = n - p - q > 1$.
Then in the presence of a single large population
correlation coefficient $\rho$ between the two groups, asymptotically
as $\rho\to 1$,
\begin{equation}
\ell_1 \approx c_1 F_{q,\nu+1}^\chi(c,n) + c_2 F_{p-1,\nu+2} + c_3
\label{eq:L1_correlation}
\end{equation}
with $c = \rho^2/(1-\rho^2)$ and
\begin{equation}
c_1 = \frac{q}{\nu+1},\quad
c_2 = \frac{p-1}{\nu+2},\quad
c_3 = \frac{p-1}{\nu(\nu-1)}.
\label{eq:c123_correlation}
\end{equation}
\end{proposition}
Comparison to the Sattherthwaite type approximation of the distribution
of $\ell_1$ due to \cite{gurl68} appears in the Supplement.

When $p=1$, the quantity $r_1^2$ reduces to the squared multiple
correlation coefficient, or coefficient of determination, between a
single `response' variable and $q$ `predictor' variables.
Eq.  \eqref{eq:L1_correlation} then reduces to a single term
$(q/(n-q)) F_{q,n-q}^\chi(c,n)$, which is in fact the \textit{exact}
distribution of $r_1^2$ in this setting, \citep[p. 173]{Muirhead_book}.


By Eq.  \eqref{eq:L1_correlation},  the largest empirical canonical correlation coefficient is biased upwards, 
\begin{equation}
  \label{eq:cancorrmean}
    \E[\ell_1] \approx
   \frac{n}{n-p-q-1} \frac{\rho^2}{1 - \rho^2}
   + \frac{p+q-1}{n-p-q-1},
\end{equation}
by both a multiplicative factor \(n/(n-p-q-1)\), and
 an additive factor. This bias may be significant for small sample sizes.


\section{Simulations}
\label{s:simulations}

We present a series of simulations that support our theoretical
analysis and illustrate the accuracy of our approximations.  For
different signal strengths we make 150,000 independent random
realizations of the two matrices $E$ and $H$, and record the largest
eigenvalue $\ell_1$.  

Figure \ref{fig:distribution_h1}  compares the empirical density of 
$(\ell_1(H) - \E[\ell_1])/\sigma(\ell_1)$
 in the signal
detection and multivariate anova cases to the theoretical formulas, 
\eqref{eq:h1def} and \eqref{eq:h1defMANOVA}, respectively.
Note that in this simulation, where all parameter values are small,
the theoretical approximation
is remarkably accurate, far more so than the classical
asymptotic Gaussian approximation. 
The latter would be valid in the large parameter limit with
$m, n_H \to \infty$ and $m/n_H \to c > 0$,
so long as $\lambda_H > \sqrt c$
(e.g. \citet{bbap05,Paul}).

\begin{figure}[t]
\includegraphics[width=0.45\columnwidth]{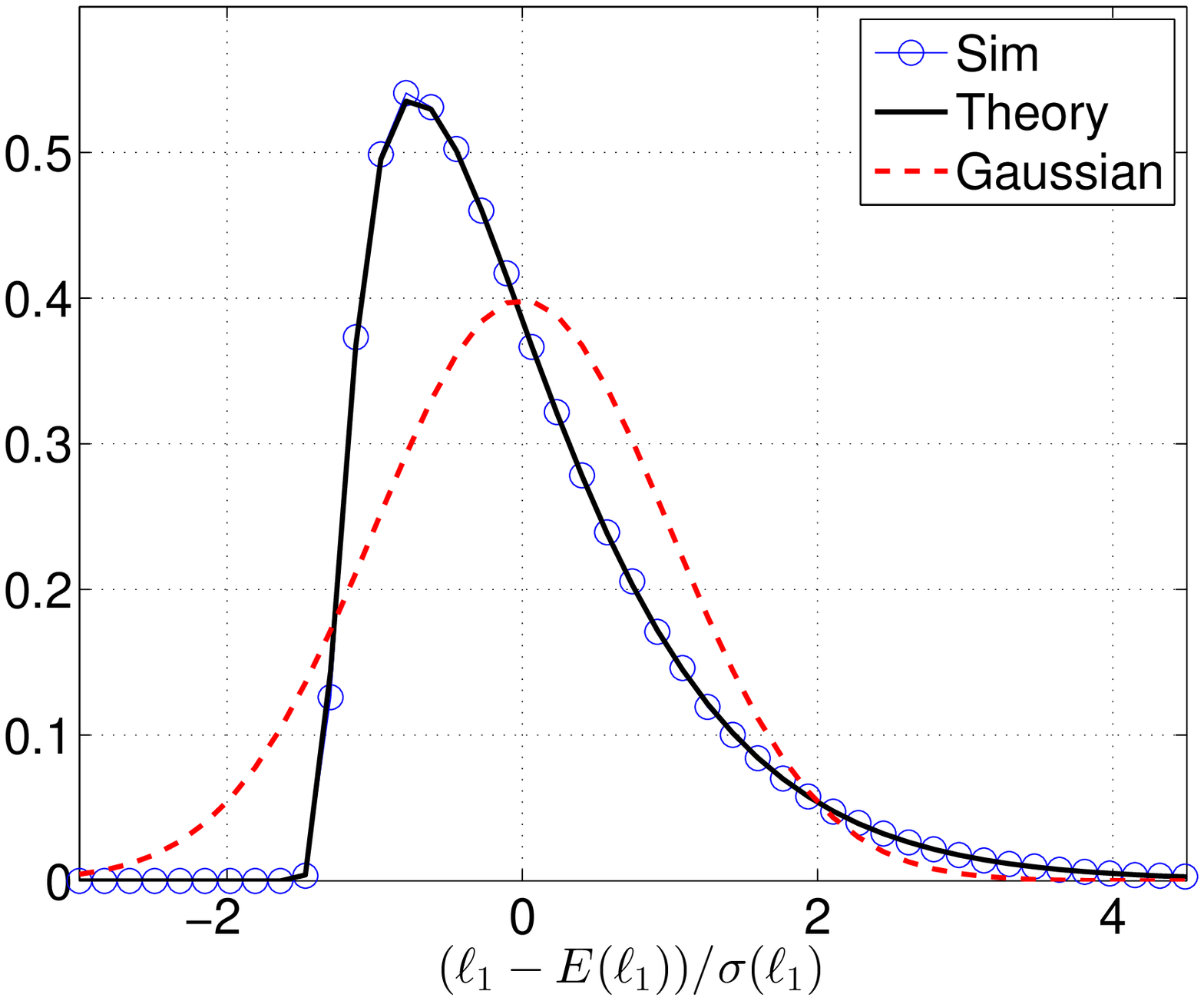}\ %
\includegraphics[width=0.45\columnwidth]{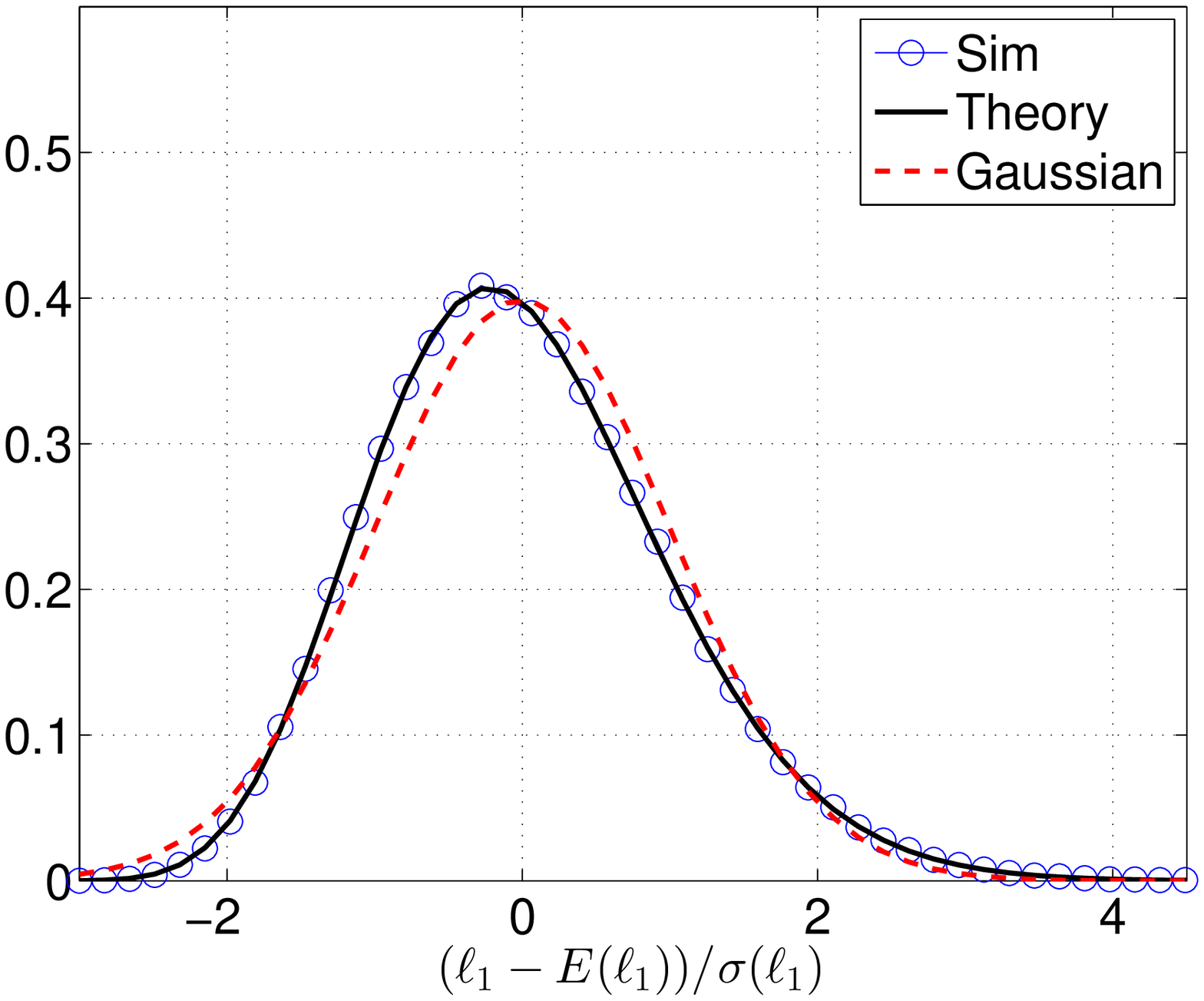}%
\caption{Density of largest eigenvalue $\ell_1(H)$. 
Left: Case 1 (signal detection), $m=5, n_H=4, \lambda_H=10, \sigma = 1$.
Right: Case 2 (Multivariate anova) $m=5$ with $p=5$ groups and $n_k = 8$
observations per group, $\omega = 40$.
Comparison of empirical density (circles) to the theoretical approximation from 
Propositions \ref{prop:h1} and \ref{prop:h1_M}
Eqs. \eqref{eq:h1def}
and \eqref{eq:h1defMANOVA}  (solid line). 
For reference, the dashed curve is the
density of a standard Gaussian.}
\label{fig:distribution_h1}
\end{figure}

\begin{figure}[ht]
\includegraphics[width=0.45\columnwidth]{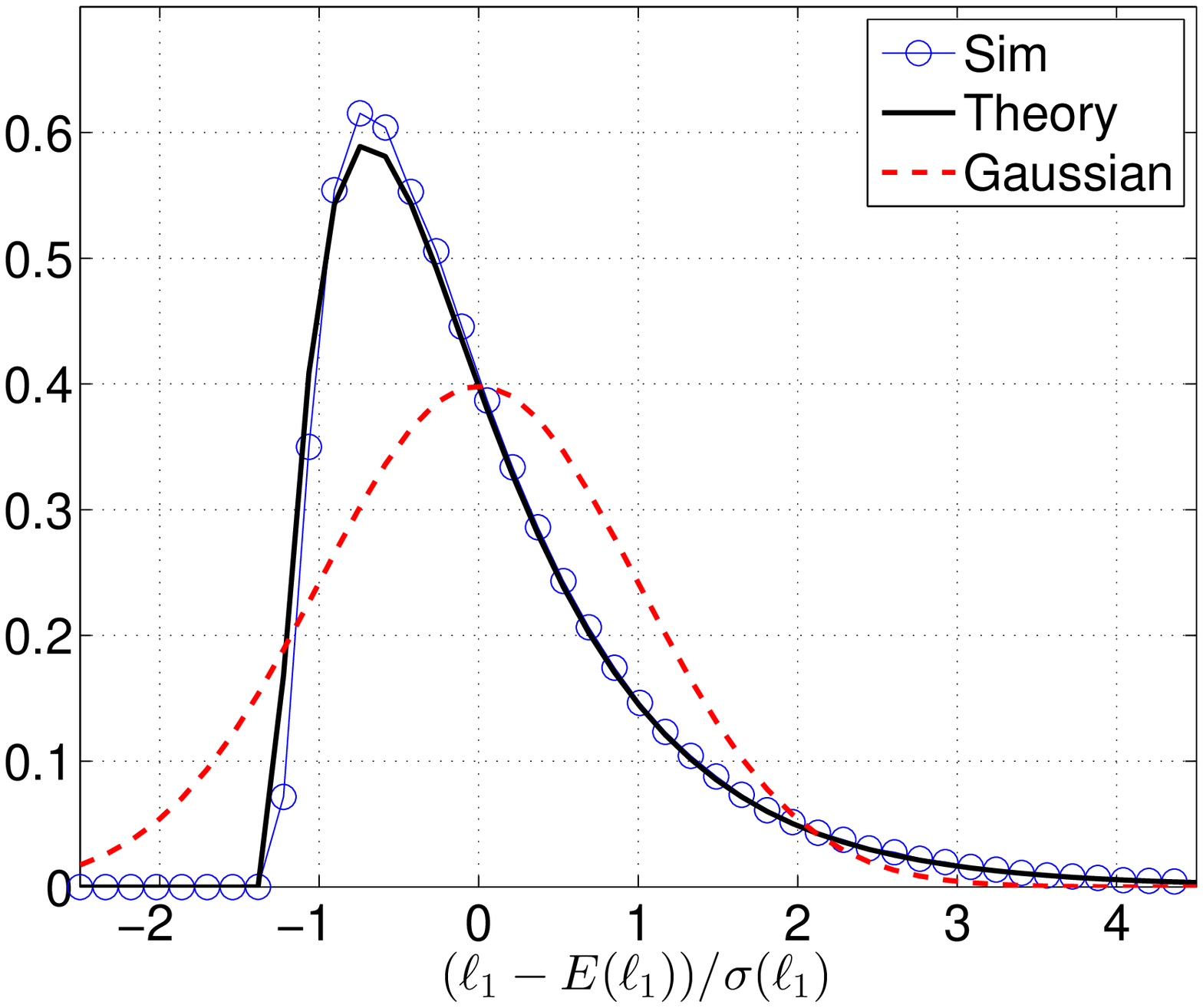}\ %
\includegraphics[width=0.45\columnwidth]{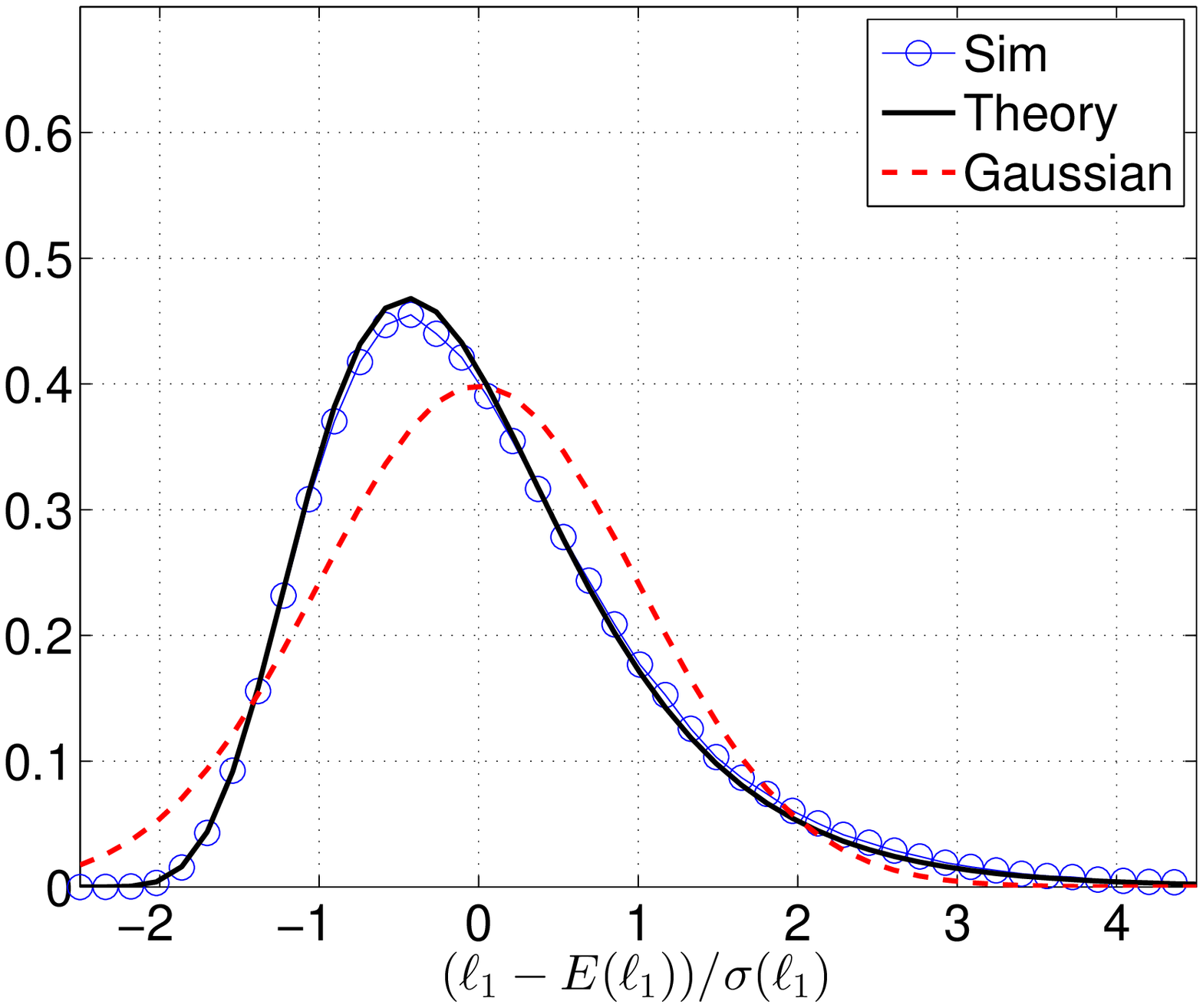}%
\caption{Density of largest eigenvalue $\ell_1(E^{-1}H)$, in
the signal detection setting, case 3 (left), and in the MANOVA setting, case 4 (right),
with $m=5, n_H=4, n_E=35$ and $\lambda_H=10$ ($\omega=40$ for case 4).
We compare the empirical density to the theoretical approximation from Propositions
\ref{prop:SP} and \ref{prop:MANOVA}, Eqs.  (\ref{eq:L_1_SP})
and \eqref{eq:ell1-approx} respectively. The dashed curve is
the density of a standard normal.  }%
\label{fig:hist_L1}%
\end{figure}

Figure \ref{fig:hist_L1} turns to the two matrix case, and the
approximate density of $\ell_1(E^{-1}H)$, after normalization,
in the signal detection and multiple anova cases 3 and 4.
Note that as expected from the analysis, the
density of the largest eigenvalue is skewed, and our
approximate theoretical distribution
is quite accurate.

\subsection{Power Calculations}

We conclude this section with a comparison
of the empirical detection probability \eqref{eq:P_D}
of Roy's test to the theoretical formulas.
We first consider the multivariate anova setting. 
Table \ref{t:power} compares the theoretical power, which follows 
from our Proposition \ref{prop:MANOVA}, 
to the results of simulations. Each entry
in the table is the result of 200,000 independent random realizations
of matrices $H$ and $E$, with $\Sigma=I$ and different non-centrality values $\omega$. 
The parameters in the table are a subset of
those studied by \citet{Olson_74}. 
Two features are apparent from the table: \\
(a) 
our approximations are quite accurate for
small sample size and dimension, and become less accurate as the
dimension increases. This is to be expected given that the leading
error terms in our expansion are of the form $O(\sqrt{m})$. \\
(b) the approximation is
relatively more accurate at high powers, say larger than $80\%$, which
fortunately
are those most relevant to design of studies in practice. 
This too is expected, as our approximation is based on high
signal-to-noise ratio, and is valid when  no eigenvalue cross-over has
occurred, meaning that the largest eigenvalue is not due to large
fluctuations in the noise. At the other extreme,
when the signal strength is weak, our approximation of power is usually
conservative since we do not model the case where the largest
eigenvalue may arise due to large deviations of the noise.

Finally, we consider setting 5 of canonical correlation analysis. The
corresponding 
comparison of simulations to theory is reported in table \ref{t:power_CCA}, with
similar behavior to Table 3.
For simulation results for the case of
detection of signals in noise, we refer to \cite{Nadler_SSP11}.

\begin{table}%
\begin{center}
\def~{\hphantom{0}}
\caption{Power of Roy's test for MANOVA}
\label{t:power}
{%
\begin{tabular}{cccccccc}
dim. & groups & samples per & non-centrality  & $P_D$ sim. & $P_D$  & $P_D$ sim. &  $P_D$\\
$m$ & $p$ & group, $n_k$ & $\omega$  & ($\alpha=1\%$) & theory  &  $\alpha=5\%$ &  theory\\[5pt]
   3 &    3 &   10 &   10 & 0.283 &  0.271 &  0.544 &  0.533 \\
   3 &    3 &   10 &   20 & 0.678 &  0.679 &  0.882 &  0.884 \\
   3 &    3 &   10 &   40 & 0.975 &  0.977 &  0.997 &  0.997 \\[3pt]
   6 &    3 &   10 &   10 & 0.150 &  0.138 &  0.369 &  0.339 \\
   6 &    3 &   10 &   20 & 0.441 &  0.428 &  0.718 &  0.704 \\
   6 &    3 &   10 &   40 & 0.875 &  0.879 &  0.975 &  0.975 \\[3pt]
   6 &    6 &   10 &   10 & 0.104 &  0.064 &  0.274 &  0.186 \\
   6 &    6 &   10 &   20 & 0.357 &  0.308 &  0.613 &  0.554 \\
   6 &    6 &   10 &   40 & 0.850 &  0.839 &  0.956 &  0.951 \\[3pt]
  10 &    6 &   20 &   10 & 0.083 &  0.054 &  0.229 &  0.143 \\
  10 &    6 &   20 &   20 & 0.312 &  0.254 &  0.551 &  0.456 \\
  10 &    6 &   20 &   40 & 0.828 &  0.795 &  0.940 &  0.917 
\end{tabular}}
\end{center}
\end{table}


\begin{table}%
\begin{center}
\def~{\hphantom{0}}
\caption{Power of Roy's test for Canonical Correlation Analysis}
\label{t:power_CCA}
{%
\begin{tabular}{cccccccc}
$p$ & $q$ & $n$    & $\rho$  & $P_D$ sim. & $P_D$  & $P_D$ sim. &  $P_D$\\
    &     &          &         & $\alpha=1\%$ & theory  &  $\alpha=5\%$ &  theory\\[5pt]
   2 &    5 &   40 & 0.50 & 0.344 &  0.336 &  0.596 &  0.578 \\ 
   2 &    5 &   40 & 0.60 & 0.653 &  0.649 &  0.849 &  0.842 \\ 
   2 &    5 &   40 & 0.70 & 0.918 &  0.917 &  0.978 &  0.977 \\ [3pt]
   3 &    7 &   50 & 0.50 & 0.313 &  0.278 &  0.565 &  0.514 \\ 
   3 &    7 &   50 & 0.60 & 0.643 &  0.618 &  0.842 &  0.822 \\ 
   3 &    7 &   50 & 0.70 & 0.925 &  0.921 &  0.980 &  0.979 \\[3pt]
   5 &   10 &   50 & 0.50 & 0.135 &  0.085 &  0.327 &  0.222 \\ 
   5 &   10 &   50 & 0.60 & 0.351 &  0.289 &  0.603 &  0.523 \\ 
   5 &   10 &   50 & 0.70 & 0.723 &  0.689 &  0.889 &  0.866  
\end{tabular}}
\end{center}
\end{table}

\section{Discussion}
\label{s:summary}


The typical approach in classical statistics studies the asymptotics of the
random variable of interest as sample size $n_H\to\infty$.
Propositions \ref{prop:h1}-\ref{prop:CCA}, in contrast, keep
$n_H,n_E,m$ \textit{fixed} but let
$\lambda_H\to\infty$, or equivalently $\sigma\to 0$.
If the signal strength is
sufficiently large, by their construction and as verified in the simulations, Propositions \ref{prop:h1}-\ref{prop:CCA}
are quite accurate for small dimension and sample size values. On the other hand, the
error in these approximations increases with the dimensionality
\(m\), and so  may not be suitable in high
dimensional small sample settings.

Whereas in this paper we focused on real-valued data, each of our settings has a complex-valued
analogue, with corresponding applications in signal processing and
communications, see \cite{dharmawansa2014roy}.

Next, we mention some directions for future research. 
The study of the distribution of Roy's largest root test under
higher dimensional alternatives is a natural extension, though it is
to be expected that the test will be less powerful than competitors
there. 
It should be possible to 
study the resulting distribution under
say two strong signals, or perhaps one strong signal and several weak ones.
Sensitivity of the distributions
to departures from normality is important.
Finally, our approach can be applied to study other test statistics, such
as the Hotelling-Lawley trace.
In addition, the approach can also provide information about
eigenvector fluctuations.

\section*{Acknowledgements}
It is a pleasure to thank
Donald Richards and David Banks for many useful discussions
and suggestions and Ted Anderson for references.
Part of this work was performed while the second author was
  on sabbatical at the U.C. Berkeley 
and Stanford Departments of Statistics, and also 
during a visit by both authors
to the Institute of Mathematical Sciences, National University of
Singapore.
The research was supported in part by NSF, NIH and BSF.

\appendix

\renewcommand{\thelemma}{\Alph{section}\arabic{lemma}}

\section{Proofs of Propositions \ref{prop:h1} and \ref{prop:h1_M}}
\label{sec:appendix}

To give the flavour of the small noise
approximation method, we present here the proofs for Propositions \ref{prop:h1} and \ref{prop:h1_M}. Remaining details, including the proofs for
Propositions 3 to 5, are in the Supplement.

We begin with a deterministic auxiliary lemma
about the change in the leading eigenvalue of a rank one
matrix due to a perturbation.
Let $\{ x_i \}_{i=1}^n$ be \(n\) vectors in $\mathbb{R}^m$ of the form
  \begin{equation}
    \label{eq:xidef}
    x_i = u_i e_1 + \epsilon \mathring{\xi}_i
  \end{equation}
with vectors $\mathring{\xi}_i^T = [0 \ \xi_i^T]$
orthogonal to $e_1^T = [1 \ 0]$; thus
$\xi_i \in \mR^{m-1}$. Let
\begin{equation}
  \label{eq:5.1}
  z = \sum_1^n u_i^2 >0, \quad
  b = z^{-1/2} \sum_1^n u_i \xi_i, \quad
  Z =  \sum_1^n \xi_i \xi_i^T.
\end{equation}
The sample covariance matrix $H = \sum_1^n x_i x_i^T$ then
admits the following decomposition
\begin{equation}
  \label{eq:A-decomp1}
H = A_0 + \epsilon A_1 + \epsilon^2 A_2
\end{equation}
where
\begin{equation}
  \label{eq:A-decomp2}
   A_0
     =
    \begin{bmatrix}
      z & \ 0 \\ 0 & \ 0
    \end{bmatrix},
\qquad
  A_1 = \sqrt{z} \begin{bmatrix}
      0 & \ b^T \\ b & \ 0
    \end{bmatrix},
\qquad
  A_2 = \begin{bmatrix}
      0 & \ 0 \\ 0 & \ Z
    \end{bmatrix}.
\end{equation}
Since $A_0,A_1$ and $A_2$ are all symmetric,
standard results from perturbation theory of linear operators
\citep{Kato} imply  that the largest eigenvalue
$\ell_1$ of $H$ and its corresponding eigenvector ${ v}_1$ are analytic
functions of $\epsilon$, near $\epsilon=0$. Specifically, with the proof appearing below, 
we establish

\begin{lemma} \label{L:perturbation}
Let $x_i$ satisfy \eqref{eq:xidef} and 
$\ell_1(\epsilon)$ be the largest eigenvalue of
$H = \sum_1^n x_i x_i^T$.
Then $\ell_1(\epsilon)$ is an even analytic function of $\epsilon$ and its
Taylor expansion around $\epsilon = 0$ is
\begin{equation}
  \label{eq:5.2}
  \ell_1(\epsilon) = z + b^T b \epsilon^2
   + z^{-1} b^T (Z - b b^T ) b \epsilon^4
   + \ldots.
\end{equation}
\end{lemma}

\subsection{Proof of Propositions \ref{prop:h1} and \ref{prop:h1_M}}
\label{s:proofs}

We now establish Propositions \ref{prop:h1} and
\ref{prop:h1_M}, assuming $\lambda_H$ or $\omega$ fixed
and $\sigma$ small.
First note that an
orthogonal transformation of the variables does not change the
eigenvalues, and so we may assume that $v = e_1$.
Thus the sum of squares matrix $H$ may be realized from $n = n_H$
i.i.d. observations \eqref{eq:xidef} with $\epsilon = \sigma$ and
\begin{equation}
  \label{eq:dist1}
  \xi_i \stackrel{\text{ind}}{\sim} N(0, I_{m-1}), \qquad
  u_i \stackrel{\text{ind}}{\sim}
  \begin{cases}
    N(0, \sigma^2 + \lambda_H) &  \text{Case 1} \\
    N(\mu_i, \sigma^2)      &   \text{Case 2},
  \end{cases}
\end{equation}
with $\sum \mu_i^2 = \omega$
and $(\xi_i)$ and $(u_i)$ independent of each other.

Lemma \ref{L:perturbation} yields the series approximation
\eqref{eq:5.2} for each realization of $u = (u_i)$ and $\Xi = [
\xi_1, \ldots, \xi_n] \in \mR^{(m-1) \times n}$ (see Supplement Section \ref{sec:remark-pert-expans} for a technical comment, on this).  Now, to clearly see the implications of the
distributional assumptions \eqref{eq:dist1}, we first rewrite \eqref{eq:5.2} in a more convenient form.

Still treating $u$ as fixed, 
define $o_1 = u/ \| u \| \in \mR^{n}$ and
then choose columns $o_2, 
\ldots, o_n$ so
that $O = [o_1 \cdots o_n]$ is an $n \times n$
orthogonal matrix. 
Let $W = \Xi O$ be $(m-1) \times n$; 
by construction, the first column of $W$
satisfies $w_1 = \Xi u/ \| u \| = b$.
Hence the $O(\epsilon^2)$ term has coefficient $b^T b =
\| w_1 \|^2$.
For the fourth order term, observe that
$Z = \Xi \Xi^T = W W^T$ and so
\begin{align*}
  D & = b^T (Z - b b^T) b
      = w_1^T (W W^T - w_1 w_1^T) w_1 
     = w_1^T \Bigl( \sum_{j=2}^n w_j w_j^T \Bigr) w_1
      = \sum_{j=2}^n (w_j^T w_1)^2.
\end{align*}
Hence \eqref{eq:5.2} becomes
\begin{displaymath}
  \ell_1 = \| u \|^2
         + \| w_1 \|^2 \epsilon^2
         + \| u \|^{-2} D \epsilon^4 + \ldots.
\end{displaymath}

Next bring in distributional assumptions
\eqref{eq:dist1} now with $n = n_H$.
First, observe that
\begin{displaymath}
  \| u \|^2 \sim
  \begin{cases}
    (\lambda_H + \sigma^2) \chi_{n_H}^2  &  \text{Case 1} \\
    \sigma^2 \chi_{n_H}^2 (\omega/\sigma^2) & \text{Case 2}.
  \end{cases}
\end{displaymath}
Since the matrix $O$ is orthogonal, and fixed once $u$ is given, the columns
$w_j \mid u \stackrel{\text{ind}}{\sim} N(0, I_{m-1})$.
As this latter distribution
does not depend on $u$, we conclude that
$\| w_1 \|^2 \sim \chi_{m-1}^2$ independently of $\| u \|^2$.
Finally, conditional on $(u, w_1)$, we have
$w_1^T w_j \stackrel{\text{ind}}{\sim} N(0, \| w_1 \|^2)$ and so
\begin{displaymath}
  D \mid (u,w_1) \sim \| w_1 \|^2 \chi_{n_H-1}^2
     \sim \chi_{m-1}^2 \cdot \chi_{n_H-1}^2,
\end{displaymath}
where the $\chi_{n_H-1}^2$ variate is independent of
$(u,w_1)$. 
This completes the proof of Propositions \ref{prop:h1} and
\ref{prop:h1_M} for the case $\sigma \to 0$.

The version of Proposition \ref{prop:h1} for $\sigma^2$ fixed and
$\lambda_H$ large is obtained by defining
$\tilde H = \lambda_H^{-1} H \sim W_m(n_H, \tilde \sigma^2 I +
v v^T)$ and applying the version just proved,
with $\epsilon^2= \tilde \sigma^2 = \sigma^2 /\lambda_H$ small. 
Similarly, the large $\omega$ version of Proposition \ref{prop:h1_M}
is obtained from the small $\sigma^2$ version by setting
$\tilde H = \omega^{-1} H \sim W_m(n_H, \tilde \sigma^2 I, \tilde
\sigma^{-2} v v^T)$ with
$\tilde \sigma^2 = \omega^{-1}$ (and now setting $\omega = 1$ in Proposition \ref{prop:h1_M}).

\subsection*{Proof of Lemma \ref{L:perturbation}}
\label{sec:proof}

First, we show that $\ell_1(\epsilon)$ is even in $\epsilon$.
Write $X(\epsilon)^T = [ x_1 \, \cdots \, x_n]$ and observe that
$X(-\epsilon)^T = U X(\epsilon)^T$ where $U = \text{diag}(1, -1, \ldots,
-1)$ is orthogonal.
Thus $H(-\epsilon) = U H(\epsilon) U^T$ and so the largest eigenvalue
$\ell_1$ and its corresponding eigenvector $v_1$ satisfy
\begin{equation}
  \label{eq:eps-prop}
  \ell_1(-\epsilon) = \ell_1(\epsilon),  \qquad
  v_1(-\epsilon) = U v_1(\epsilon).
\end{equation}
Thus $\ell_1$ and the first component of $v_1$ are even functions
of $\epsilon$ while the other components of $v_1$ are odd.

Now expand $\ell_1$ and $v_1$ in a Taylor series in
$\epsilon$,
using \eqref{eq:eps-prop} and $\lambda_{2k-1}=0$:
\begin{align*}
\ell_1 & = \lambda_0 +  \epsilon^2 \lambda_2
+ \epsilon^4 \lambda_4 + \ldots\\
v_1 & = w_0 + \epsilon w_1 + \epsilon^2 w_2 +
\epsilon^3 w_3 + \epsilon^4 w_4 + \ldots
\end{align*}
Inserting this expansion into the eigenvalue equation $Hv_1 =
\ell_1v_1$ gives the following set of equations  
\begin{equation}
  \label{eq:e-eqns}
  A_0 w_{r} + A_1 w_{r-1} + A_2 w_{r-2}
   = \lambda_0 w_{r} + \lambda_2 w_{r-2} + \lambda_4
   w_{r-4} + \cdots,\quad
   r=0,1,2,\ldots
\end{equation}
with the convention that vectors with negative subscripts are zero.
From the $r=0$ equation, $ A_0 w_{0} = \lambda_0 w_{0}$,
we readily find that
\begin{displaymath}
\lambda_0 = z, \qquad w_0 = { e}_1.
\end{displaymath}
Since the eigenvector $v_1$ is defined up to a normalization
constant, we choose it such that
$v_1^T { e}_1 = 1$ for all $\epsilon$. This
implies that $w_j$, for $j\geq 1$, are all orthogonal to ${
  e}_1$, that is, orthogonal to $w_0$.

From the eigenvector remarks following \eqref{eq:eps-prop} it follows
that $w_{2k} = 0$ for  $k \geq 1$.
These remarks allow considerable simplification of equations
\eqref{eq:e-eqns}; we use those for $r = 1$ and $3$:
\begin{equation}
  \label{eq:simpler}
     A_1 w_0 = \lambda_0 w_1, \qquad
     A_2 w_1 = \lambda_0 w_3 + \lambda_2 w_1,
\end{equation}
from which we obtain,
on putting $\mathring{b}^T = [0 \ b^T]$,
\begin{displaymath}
  w_1 = z^{-1/2} \mathring{b}, \qquad
  w_3 = \lambda_0^{-1} (A_2 - \lambda_2 I) w_1.
\end{displaymath}
Premultiply \eqref{eq:e-eqns} by $w_0^T$ and use the first
equation of \eqref{eq:simpler} to get, for $r$ even,
\begin{displaymath}
  \lambda_r = (A_1 w_0)^T w_{r-1} = \lambda_0 w_1^T
  w_{r-1},
\end{displaymath}
and hence
\[
  \lambda_2    = \lambda_0 w_1^T w_1 = b^T b, 
  \quad
    \lambda_4  = w_1^T (A_2 - \lambda_2 I) w_1
             = z^{-1} b^T (Z - b b^T)
             b. \qquad  
\]

\bibliographystyle{biometrika}
\bibliography{paper-ref,Roys_bib}


\newpage
\begin{center}
  \Large \textbf{Supplementary Materials}
\end{center}

    \numberwithin{equation}{section}
\setcounter{lemma}{0}
\renewcommand{\thelemma}{\Alph{section}\arabic{lemma}}

\section{Sattherthwaite type approximations}
\label{sec:suppl-mater}

A Sattherthwaite type approximation to the distribution
    of the multiple correlation coefficient was given by \citet{gurl68}, see
    also Muirhead, pp. 176-7. Using our $F^\chi$ terminology, we
    approximate $\chi_a^2(Z)$ in \eqref{eq:XF_distribution} by a
    scaled gamma variate, written formally as $g \chi_f^2$ with
    non-integer $f$. Equating the first two moments yields
    \begin{displaymath}
      g = \frac{cn(c+2) + a}{cn + a}, \qquad
      f = \frac{cn + a}{g}.
    \end{displaymath}
In the setting of Proposition \ref{prop:CCA}, we then approximate
\begin{displaymath}
  c_1  F_{q,\nu+1}^\chi(c,n) \approx g F_{f,\nu +1},
\end{displaymath}
with $a = q$ and $(c, n)$ as in the Proposition.
Gurland provides limited numerical evidence that this approximation is
adequate in the near right tail needed for power calculations so long
as $c$ is moderate.

\section{Further simulations}
\label{sec:further-simulations}

First, in  Fig. \ref{fig:E_L1}
we compare the empirical mean of both $\ell_1(H)$ and of $\ell_1(E^{-1}H)$
to the theoretical formulas, 
\eqref{eq:case1-mean}, and \eqref{eq:ell1mean}.
%
Next, in
the left panel of Fig. \ref{fig:V_L1}  we compare the standard deviation
$\sqrt{var[\ell_1]}$ for both the MANOVA and the signal detection case
to the theoretical formulas, \eqref{eq:cases12-var}. 
Finally, in the right panel of
Fig. \ref{fig:V_L1} we compare the standard deviation of Roy's largest
root test in the two settings to the theoretical predictions based on
Propositions \ref{prop:SP} and \ref{prop:MANOVA}
and the variants of formula \eqref{eq:ell1var}.
Note that in this simulation all parameter values are small ($m=5$
dimensions, $p=5$ groups with $n_i=8$ observations per group yielding
a total of $n=40$ samples), and the fit between the simulations and theory
is quite good.


\begin{figure}[ht]%
\begin{center}
\includegraphics[width=0.49\columnwidth]{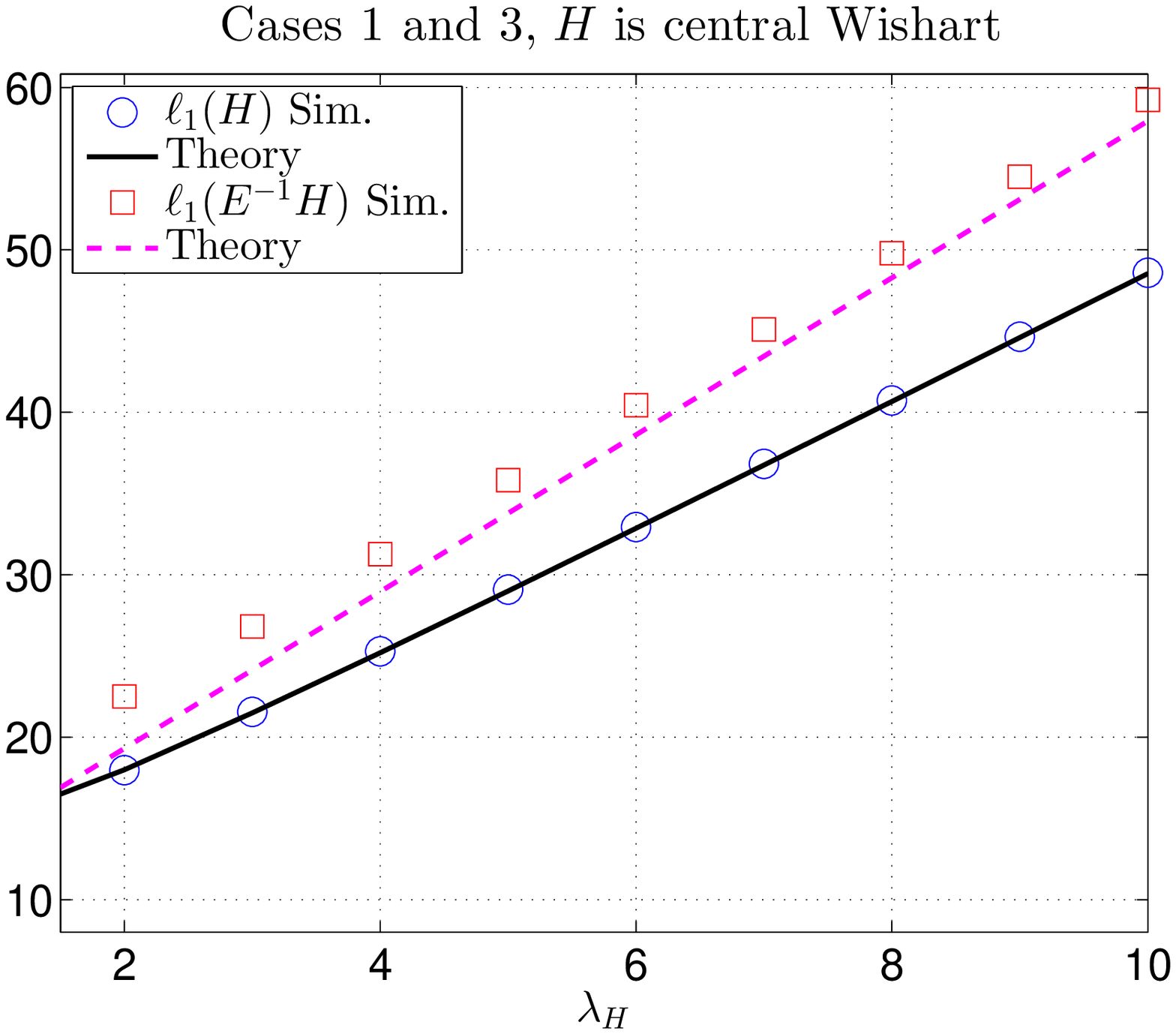}
\
\includegraphics[width=0.49\columnwidth]{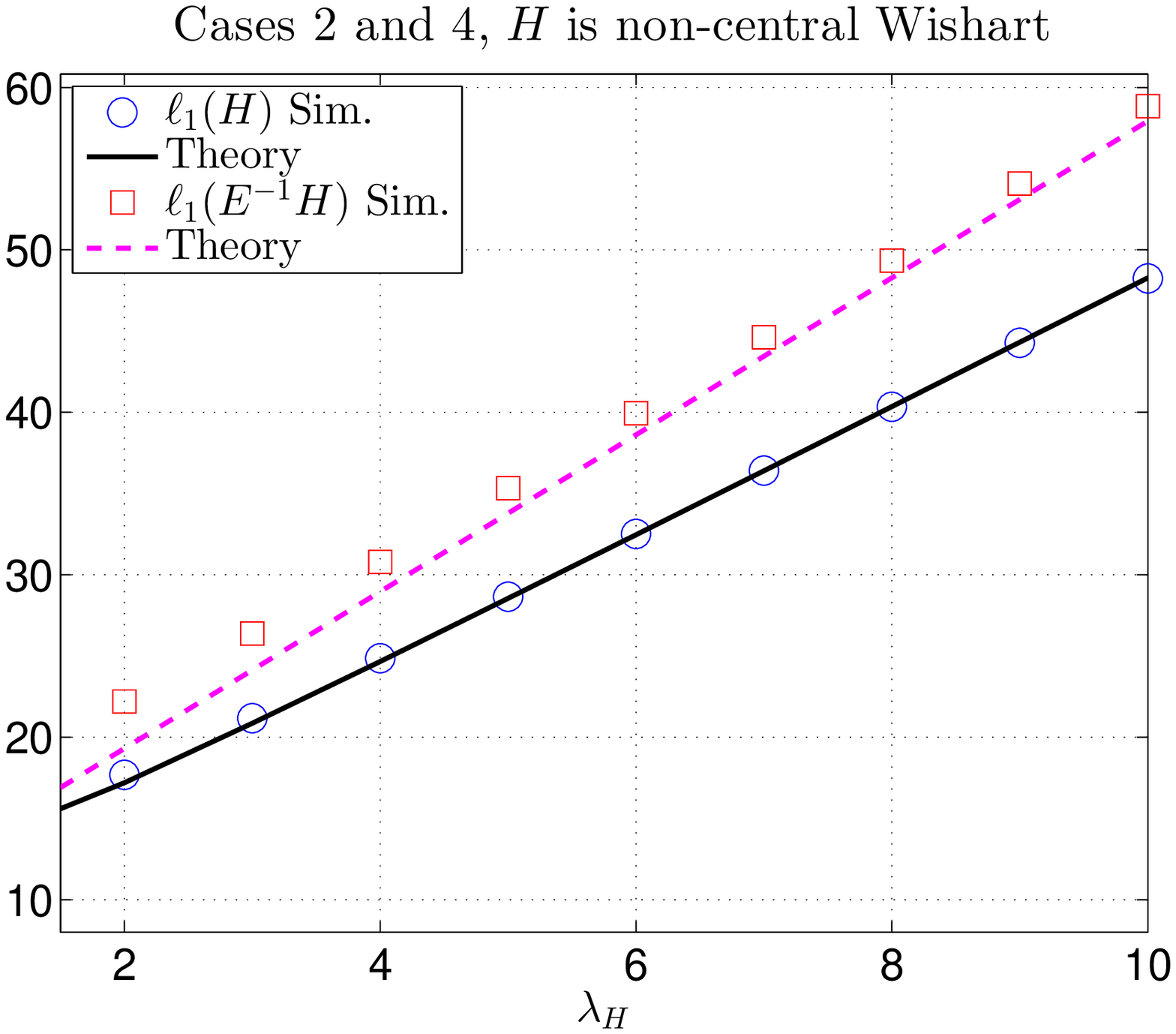}
\caption{Mean of the largest eigenvalue of $H$ and of $n_{E}E^{-1}H$
in both the signal detection setting (Cases 1 and 3, left panel) and in MANOVA (Cases 2 and 4, right panel).
In both simulations, $n_H=4, n_E=35$, $m=5$ and $\sigma=1$. In the MANOVA case, $\omega = \lambda_H n_H$. }%
\label{fig:E_L1}%
\end{center}
\end{figure}

\begin{figure}[ht]
\begin{center}
\includegraphics[width=0.45\columnwidth]{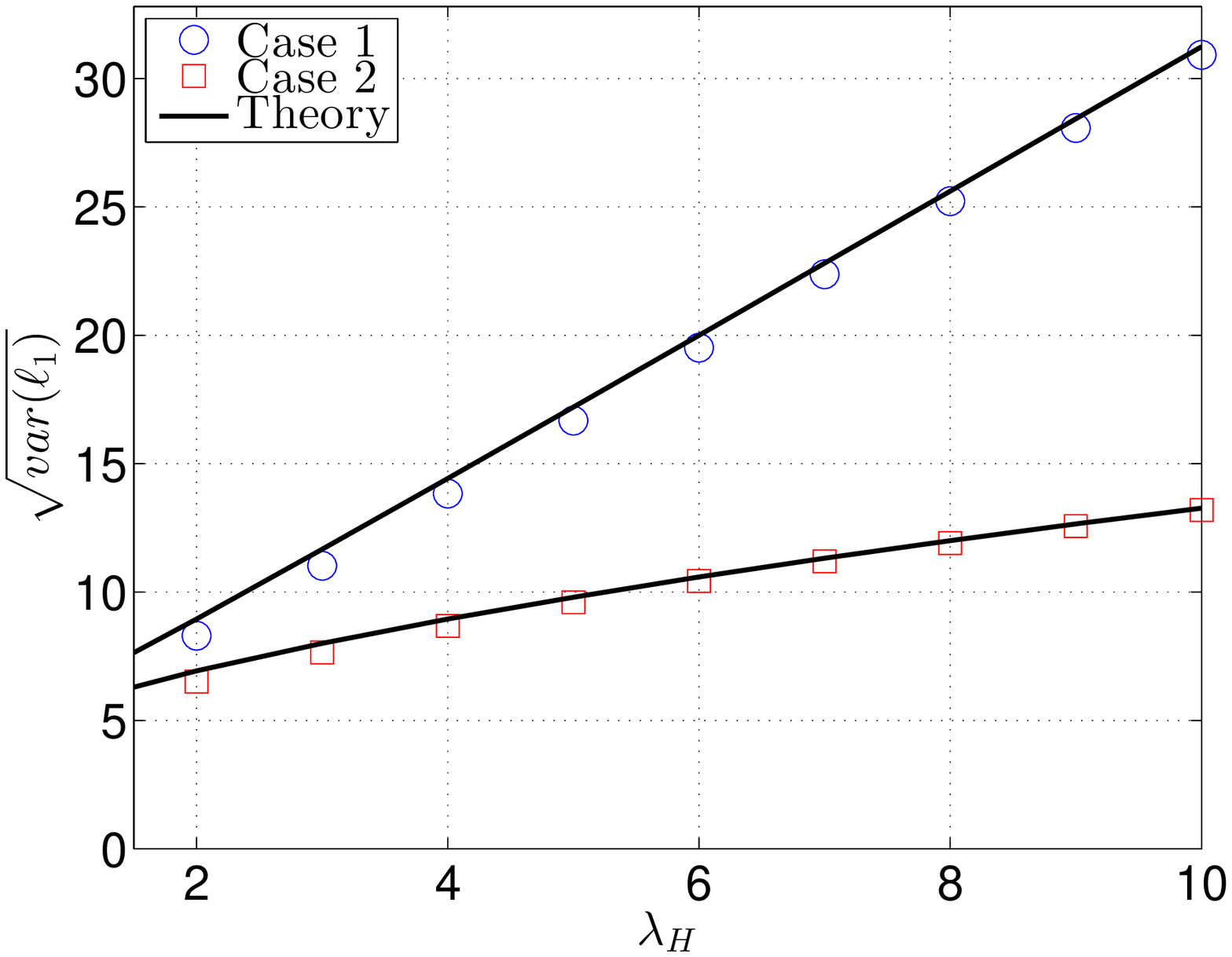}\
\includegraphics[width=0.45\columnwidth]{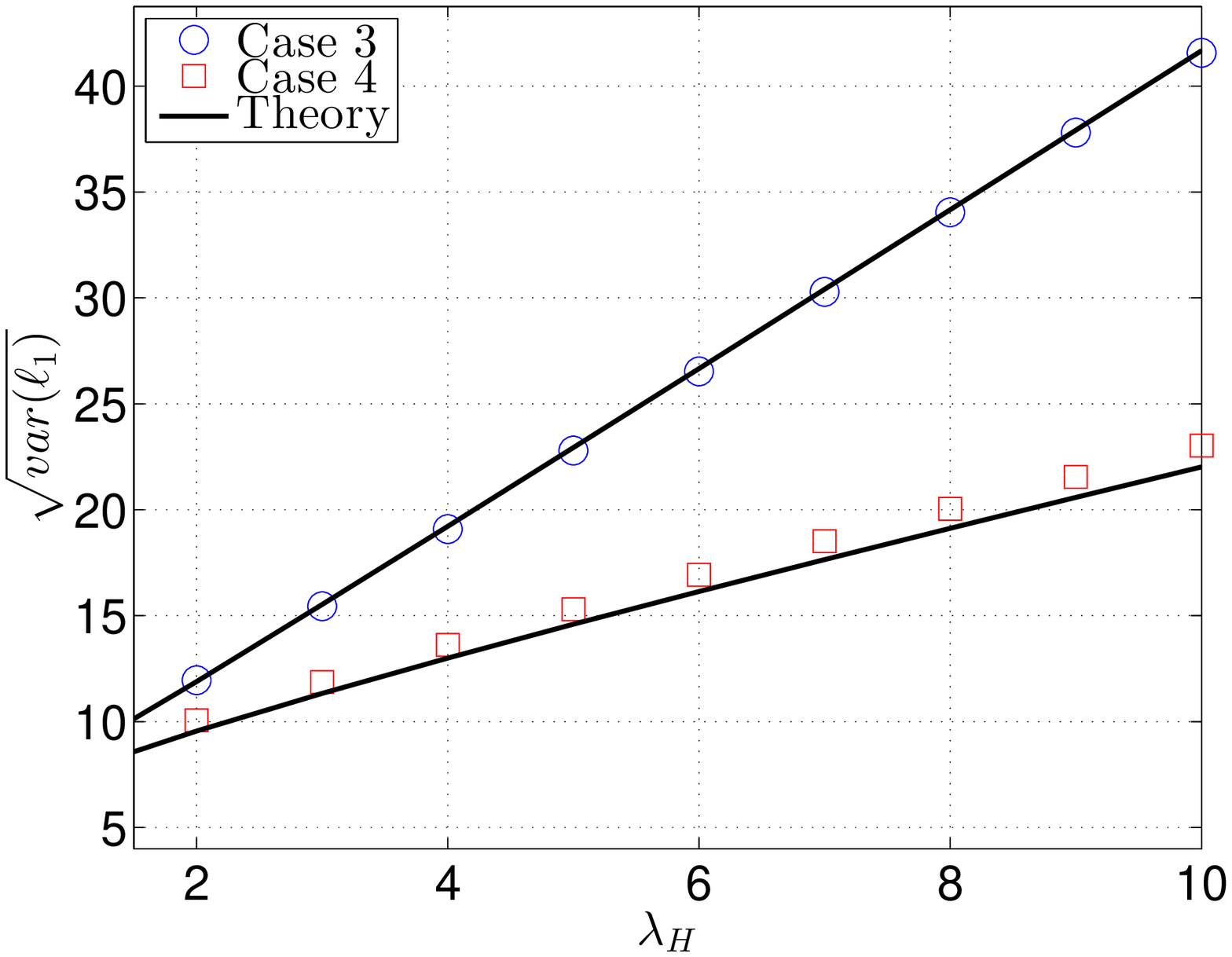}%
\caption{Standard deviation of the largest eigenvalue of \(H\) in settings 1 and 2 (left), 
and  of  $\ell_1(n_EE^{-1}H)$ in settings 3 and 4 (right). 
Comparison of simulations results to theoretical approximations. }%
\label{fig:V_L1}%
\end{center}
\end{figure}

\section{Proof of Propositions \ref{prop:SP} and \ref{prop:MANOVA}}
\label{sec:Proof_3_4}

First, from the similarity equation
\begin{equation*}
  E^{-1} H 
  = \Sigma^{-1/2} (\Sigma^{-1/2} E \Sigma^{-1/2})^{-1} 
                         (\Sigma^{-1/2} H \Sigma^{-1/2}) 
                       \Sigma^{1/2},
\end{equation*}
it follows that the matrix $E^{-1}H$ has the same set of eigenvalues
as does $(\Sigma^{-1/2} E \Sigma^{-1/2})^{-1}(\Sigma^{-1/2} H
\Sigma^{-1/2})$.
So, we may assume without loss of generality that $\Sigma = I$, 
and furthermore that the signal 
direction is  $v = e_1$. Hence we
assume that $ E\sim  W_m(n_E,I)$, and that
\begin{displaymath}
  H \sim
  \begin{cases}
    W_m(n_H,I + \lambda_H e_1e_1^T) & \text{Case 3 (SD)} \\
    W_m(n_H, I, \ \omega e_1e_1^T) & \text{Case 4 (MANOVA)}.
  \end{cases}
\end{displaymath}

Next, we apply a perturbation approach similar to the one
used in proving the first two propositions.
To introduce a small parameter, set
\begin{displaymath}
  \epsilon^2 =
  \begin{cases}
    1/(1+\lambda_H) & \text{SD} \\
    1/\omega      & \text{MANOVA}.
  \end{cases}
\end{displaymath}
The matrix $H_\epsilon = \epsilon^2 H$ has a representation of the
form $X^T X$, where the matrix $X = [x_1 \cdots x_{n_H}]$ 
and each $x_i$ is of the
form \eqref{eq:xidef}, but now with
\begin{equation}
  \label{eq:uidef}
  \xi_i \stackrel{\text{ind}}{\sim} N(0, I_{m-1}), \qquad
  u_i \stackrel{\text{ind}}{\sim}
  \begin{cases}
    N(0, 1) &  \text{SD} \\
    N(\mu_i/\sqrt \omega, 1/\omega)      &   \text{MANOVA},
  \end{cases}
\end{equation}
with $\sum \mu_i^2 = \omega$.
In particular,  $H_\epsilon$ has decomposition
\eqref{eq:A-decomp1}--\eqref{eq:A-decomp2}, where
\begin{equation}
  \label{eq:z-dist}
  z = \sum_{i=1}^{n_H} u_i^2
    \sim
    \begin{cases}
      \chi_{n_H}^2  & \text{SD} \\
      \omega^{-1} \chi_{n_H}^2 (\omega) & \text{MANOVA}.
    \end{cases}
\end{equation}
With $b$ as in \eqref{eq:5.1}, we have, conditional on $z$,
that $b \sim N(0,I_{m-1})$.

To apply a perturbation approximation, first note that
the eigenvalues of \(E^{-1}H_\epsilon\) are the same as those of
the symmetric matrix $E^{-1/2} H_\epsilon E^{-1/2}$, 
so it follows that the
largest eigenvalue and its corresponding eigenvector are analytic
functions in $\epsilon$, for sufficiently small \(\epsilon\), see
\cite{Kato}. 

We now define some terms appearing in the resulting series
approximation for $\ell_1(E^{-1} H)$. 
Introduce the vector $\mathring{b} =
\begin{pmatrix}
  0 \\ b
\end{pmatrix}$,
the $m \times 2$ matrix $M = [e_1 \ \mathring{b}]$
and the symmetric matrix
\begin{equation}
S^{-1} = M^T E^{-1} M =  
\begin{pmatrix}
  e_1^T E^{-1} e_1 & \mathring{b}^T E^{-1}e_1
  \\
  e_1^T E^{-1} \mathring{b}  & \mathring{b}^T E^{-1}
  \mathring{b} 
\end{pmatrix}.
        \label{eq:S_def}
\end{equation}
Here and below, for a matrix $E$, $E^{ij}$ denotes the $(i,j)$-th
entry of $E^{-1}$.
Finally, with $A_2$ as in \eqref{eq:A-decomp2}, let
\begin{equation}
  \label{eq:Rdef}
  R = e_1^T E^{-1}A_2E^{-1}e_1/ E^{11}
    = d^T Z d,
\end{equation}
where $d = P_2 E^{-1} e_1 / \sqrt{E^{11}}$ and $P_2$ is
projection on the last $m-1$ co-ordinates.

\begin{lemma}
  \label{lem:two-matrix-expansion}
With the preceding definitions, 
\begin{equation}
\label{eq:two-matrix-expansion}
\ell_1(E^{-1}H_\epsilon) 
 = z S^{11} + 2 \epsilon \sqrt{z} S^{12} 
   + \epsilon^2 (R + 1/S^{22}) + o(\epsilon^2).
\end{equation}
\end{lemma}

To discuss the individual terms in this expansion, 
we make use of two auxiliary lemmas.

\begin{lemma}\label{lemma:S}
Let $ E\sim W_m(n_E,I)$ and define $S$ as in \eqref{eq:S_def}.
Then, conditional on $b$, 
\begin{equation}
  \label{eq:display}
S 
\sim W_2(n_E-m+2,D),  \qquad
D = \text{\rm diag}(1, 1/ \|b\|^2)
\end{equation}
and the two random variables $S^{11}$ and $S_{22}$ are independent
with
\begin{equation}
S^{11} \sim \frac{1}{\chi^2_{n_E-m+1}},
\quad\quad
S_{22} \sim \frac{\chi^2_{n_E-m+2} }{\|{b}\|^2}.
\label{eq:S-dists}
\end{equation}
\end{lemma}

\begin{lemma}\label{lemma:A2} Let $ E\sim W_m(n_E,I)$, and 
define $R$ as in \eqref{eq:Rdef}. Then 
\begin{equation}
\mathbb{E} R
= \frac{(m-1)}{(n_E-m)(n_E-m-1)}.
        \label{eq:E_A2}
\end{equation}
\end{lemma}

\textit{Approximations.} \ To establish Propositions \ref{prop:SP} and
\ref{prop:MANOVA}, we start from \eqref{eq:two-matrix-expansion}.
We neglect the second term $T_1 = 2\epsilon \sqrt{z} S^{12}$ which is
symmetric with mean zero, and whose variance is much smaller than that
of the first term.
We also approximate $T_2 = \epsilon^2 R$ by its mean value
using Lemma \ref{lemma:A2}.
We arrive at
\begin{displaymath}
   \ell_1( E^{-1} H_\epsilon)
   \approx  z S^{11} + \epsilon^2 /S_{22}
            +  \epsilon^2 c(m, n_E),
\end{displaymath}
where $c(m, n_E)$ is the constant in \eqref{eq:E_A2}.
Denote the first two terms on the right side by $F(S;z,\epsilon)$.
Condition on $u$ and $b$ in the representation
\eqref{eq:xidef} 
for $H_\epsilon$; then Lemma \ref{lemma:S} tells us that conditionally
\begin{displaymath}
  \epsilon^{-2} F(S;z, \epsilon)
    \ \stackrel{\mathcal{D}}{\sim} \
    \frac{\epsilon^{-2} z}{\chi_{n_E-m+1}^2}
            + \frac{\| b \|^2}{\chi_{n_E-m+2}^2}.
\end{displaymath}
The two $\chi^2$ variates are functions of $E$ alone, hence their
distributions do not depend on $b$.
Unconditioning on $b$, we have $\| b \|^2 \sim
\chi_{m-1}^2$, and  conditional on $z$,
these three $\chi^2$ variates are jointly independent with distributions
not depending on $z$.
Finally, unconditioning on $z$, we have $z$ distributed as in
\eqref{eq:z-dist}, independent of all three $\chi^2$ variates. The
conclusions of Propositions 3 and 4 now follow.
For example, for Proposition \ref{prop:SP},
\begin{displaymath}
  \epsilon^{-2} F(S;z, \epsilon)
    \ \stackrel{\mathcal{D}}{\sim} \
    \frac{(1+\lambda_H)n_H}{n_E -m+1} F_{n_H, n_E-m+1}
    + \frac{m-1}{n_E-m+2} F_{m-1,n_E-m+2}.
\end{displaymath}

The expectation expressions \eqref{eq:ell1mean} follow from
independence  and the formulas
$ \E \chi_n^2(\omega) = n + \omega$ and
$\E [1/\chi_n^{2}] = (n-2)^{-1}$.

\medskip
\textit{Error terms.} \ In the supplementary material it is
argued---heuristically in the case
of $T_2$---that if both $m$ and $n_H \leq 2 n_E$, then
\begin{equation}
  \label{eq:Var-Ti}
  \text{var} \, T_1 = \frac{4 \epsilon^2 (m-1) E z}{\nu(\nu-1)(\nu-3)}, \qquad
  \text{var} \, T_2 \asymp \frac{ \epsilon^4 m n_H(m+n_H)}{n_E^4}.
\end{equation}
Here $L \asymp R$ means that $L/R$ is bounded above and below by
positive constants not depending on the parameters in $R$.

It is then shown there that
\begin{equation}
  \label{eq:Var-Ti1}
  \max_{i=1,2} \text{var} \, (\epsilon^{-2} T_i)
    \leq \frac{c}{n_E} \frac{m}{n_E} \frac{n_H}{n_E}
    \begin{cases}
      1 + \lambda_H  &  \text{Case 3} \\
      1 + \omega/n_H  &  \text{Case 4}.
    \end{cases}
\end{equation}
Consequently the fluctuations of the terms we ignore are typically of
smaller order than the leading terms in Propositions \ref{prop:SP} and
\ref{prop:MANOVA}; more precisely (with a different constant $c'$):
\begin{displaymath}
    \max_{i=1,2} \, \text{SD} \, (\epsilon^{-2} T_i)
    \leq  \frac{c' \sqrt{m}}{n_E} \sqrt{ \E \ell_1(E^{-1} H)}.
\end{displaymath}

\section{Proof of Proposition \ref{prop:CCA}}

The canonical correlation problem is invariant under change of basis
for each of the two sets of variables,
e.g. \citet[Th. 11.2.2]{Muirhead_book}.
We may therefore assume that the matrix $\Sigma$ takes
the canonical form
\begin{displaymath}
\Sigma  =
 \begin{pmatrix}
   I_p & \tilde P \\
   \tilde P^T & I_q
 \end{pmatrix},
\quad
\tilde P = [ P \ 0 ],
\quad
P = \text{diag}(\rho, 0, \ldots, 0)
\end{displaymath}
where $\tilde P$ is $p\times q$ and the matrix $P$ is of size $p\times
p$ with a single non-zero population
canonical correlation $\rho$.
Furthermore, in this new basis, we decompose the sample covariance matrix
as follows,
\begin{equation}
nS = 
\begin{pmatrix}
  Y^T Y & Y^T X \\
  X^T Y & X^T X
\end{pmatrix}
\label{eq:DSDT1}
\end{equation}
where the columns of the $n \times p$ matrix $Y$ contain the first $p$
variables of the $n$ samples, now assumed to have mean $0$,
represented in the transformed basis.
Similarly, the columns of $n \times q$ matrix $X$ contain the remaining $q$
variables.
For future use, we note that the
matrix $X^T X \sim W_q(n,I)$.

As noted earlier, the squared canonical correlations $\{ r_i^2 \}$ are
the eigenvalues of
$  S_{11}^{-1}   S_{12}   S_{22}^{-1}   S_{21}$.
Equivalently, if we set $P_X = X (X^T X)^{-1} X^T$
they are the roots of
\begin{displaymath}
  \det( r^2 Y^T Y - Y^T P_X Y) = 0.
\end{displaymath}

Set $H = Y^T P_X Y$ and $E = Y^T(I-P_X) Y$: the previous equation
becomes $\det(H - r^2(H+E)) = 0$.
Instead of studying the largest root of this equation,
we transform to $\ell_1 = r_1^2 /(1-r_1^2)$, the largest root of $E^{-1}H$.
We now appeal to a standard partitioned Wishart argument.
Conditional on $X$, the matrix $Y$ is Gaussian with independent rows,
and mean and covariance matrices
\begin{align*}
  M(X) & = X \Sigma_{22}^{-1} \Sigma_{21} = X \tilde P^T \\
  \Sigma_{11 \cdot 2} & = \Sigma_{11} - \Sigma_{12} \Sigma_{22}^{-1}
\Sigma_{21} = I-P^2 := \Phi.
\end{align*}
Conditional on $X$, and using Cochran's theorem, the matrices
\begin{align*}
  H  \sim W_p(q, \Sigma_{11.2},\Omega(X)) \quad\mbox{and}\quad
  E  \sim W_p(n - q, \Sigma_{11.2})
\end{align*}
are independent, where the noncentrality matrix
\begin{displaymath}
  \Omega(X) = \Sigma_{11 \cdot 2}^{-1} M(X)^T M(X)
      = \Phi^{-1} \tilde P X^T X \tilde P^T
      = c Z e_1 e_1^T,
\end{displaymath}
where $Z = (X^T X)_{11} \sim \chi_n^2$.
Thus $\Omega(X)$ depends only on $Z$.
Apply Proposition \ref{prop:MANOVA}
with $H \sim W_p(q, \Phi, \Omega(Z))$ and
$E \sim W_p(n-q, \Phi)$ so that
conditional on $X$, the distribution of $\ell_1$ is approximately
given by  \eqref{eq:ell1-approx}--\eqref{eq:c-pars} with
\begin{displaymath}
  a_1 = q, \quad
  a_2 = p - 1, \quad
  \nu = n - q - p, \quad
  \omega = \frac{\rho^2}{1-\rho^2}Z.
\end{displaymath}
Since $Z\sim\chi^2_n$, the Proposition follows from the definition of
$F^\chi$.
\hfill $\Box$

\section{Proofs of Auxiliary Lemmas}

\begin{proof}[Proof of Lemma \ref{lem:two-matrix-expansion}]

The argument is a modification of that of lemma \ref{L:perturbation} in the main text. 
For the matrix $H_\epsilon = \sum x_i x_i^T$, we adopt the
notation of \eqref{eq:xidef}--\eqref{eq:A-decomp2}.
We  expand
\begin{equation}
\ell_1(E^{-1}H_\epsilon)  
   =  \sum_{i=0}^\infty \lambda_i \epsilon^i,    \quad \quad
v_1                
   =  \sum_{i=0}^\infty w_i \epsilon^i. \nonumber
\end{equation}
Inserting these expansions into the eigenvalue-eigenvector equation
$E^{-1} H_\epsilon v_1 = \ell_1  v_1$ we get
the following equations.
At the $O(1)$ level,
\begin{displaymath}
E^{-1}A_0w_0 = \lambda_0 w_0
\end{displaymath}
whose solution is
\begin{displaymath}
\lambda_0 = z E^{11},\quad
w_0 = E^{-1}e_1.
        \label{eq:lambda0_2}
\end{displaymath}
Since the eigenvector $v_1$ is defined up to a normalization, we
choose it to be the constraint \(e_1^Tv_1 =e_1^T w_0= E^{11}\), which
implies that $e_1^Tw_j=0$ for all 
$j\geq 1$. Furthermore, since \(A_{0}=ze_1e_1^T\), this
normalization also conveniently gives that $A_0w_j=0$ for all
$j\geq 1$.

The \(O(\epsilon)\) equation is
\begin{equation}
E^{-1}A_1 w_0 +
E^{-1}A_0 w_1 =
\lambda_1 w_0 + \lambda_0 w_1.
             \label{eq:O_sigma}
\end{equation}
However, \(A_{0}w_1=0\). Multiplying this equation by $e_1^T$
 gives that
\begin{displaymath}
\lambda_1 = \frac{e_1^T E^{-1}A_1w_0}{E^{11}} = 2 \sqrt z
\mathring{b}^TE^{-1}e_1 =: 2 \sqrt z E^{b1}.
\end{displaymath}
Inserting the expression for $\lambda_1$ into Eq. (\ref{eq:O_sigma}) gives that
\begin{displaymath}
w_1 = \frac1{\sqrt z}
\left[
E^{-1} \mathring{b} -\frac{E^{b1}}{ E^{11} }E^{-1}e_1
\right].
\end{displaymath}
The next \(O(\epsilon^2)\) equation is
\begin{displaymath}
E^{-1}A_2 w_0 +
E^{-1}A_1 w_1
= \lambda_2w_0 +
\lambda_1 w_1 + \lambda_0 w_2.
\end{displaymath}
Multiplying this equation by $e_1^T$, and recalling that
$A_0w_2=0$,  gives
\begin{displaymath}
\lambda_2 = \frac{E^{11}E^{bb} - (E^{b1})^2}{E^{11}}
+ \frac{e_1^T E^{-1} A_2 E^{-1} e_1}{E^{11}}.
\end{displaymath}
Combining the previous six displays,
we obtain the required
approximate stochastic representation
for the largest eigenvalue $\ell_1(E^{-1}H_\epsilon )$.
\end{proof}

\begin{proof}[Proof of Lemma \ref{lemma:S}]

This is classical: first note that $S^{11} = E^{11}$ is a diagonal entry of the
inverse of a Wishart matrix, so Theorem 3.2.11 from \cite{Muirhead_book} yields
that $S^{11} \sim 1 / {\chi^2_{n_E-m+1}}$.

Next, by definition,
$ S = (M^T E^{-1} M)^{-1}$, with $M$
being fixed. Hence the same theorem gives
$ S\sim W_2(n_E-m+2,D)$, with $D$ as in \eqref{eq:display},
so that $S_{22} \sim \chi^2_{n_E-m+2} / \|{b}\|^2$.
Finally, the fact that $S^{11}$ and $S_{22}$ are independent follows
from Muirhead's Theorem 3.2.10.
\end{proof}


\begin{proof}[Proof of Lemma \ref{lemma:A2}]
  In the representation $R = d^T Z d$ we note that
$d$ is a function of $E$ and hence is independent of $Z \sim
W_{m-1}(n_H, I)$. 
So by conditioning on $d$, we have
\begin{equation}
  \label{eq:E12}
  \E R = n_H \E d^T d.
\end{equation}
Partition
\begin{displaymath}
  E =
  \begin{pmatrix}
    E_{11} & E_{12} \\
    E_{21} & E_{22}
  \end{pmatrix}, 
 \qquad 
  E^{-1} = 
  \begin{pmatrix}
    E^{11} & E^{12} \\
    E^{21} & E^{22}
  \end{pmatrix},
\end{displaymath}
where $E_{11}$ is scalar and $E_{22}$ is square of size $m-1$.
We have $d^T d = e_1^T E^{-1} P_2 E^{-1}
e_1 / E^{11}$ and claim that
\begin{equation}
  \label{eq:vTv}
  d^T d = \text{tr} (E^{22} - E_{22}^{-1} ).
\end{equation}
Indeed this may be verified by applying the partitioned matrix inverse
formula, e.g. MKB, p459,
to $A = E^{-1}$.
Consequently
\begin{displaymath}
  \text{tr} (E^{22} - E_{22}^{-1} )
   = \text{tr} (A_{21} A_{11}^{-1} A_{12}) 
   = A_{12} A_{21}/ A_{11}
\end{displaymath}
and we identify $A_{11}$ with $E^{11}$ and $A_{12}$ with 
$E^{12} = e_1^T E^{-1} P_2$.

Now $E_{22} \sim W_{m-1}(n_E,I)$ and 
$(E^{22})^{-1} \sim W_{m-1}(n_E -1, I)$, for example using 
MKB, Coroll. 3.4.6.1.
For $W \sim W_p(n,I)$, we have $\E W^{-1} = (n-p-1)^{-1} I$, e.g.
\citet[p. 97]{Muirhead_book}, and so Lemma \ref{lemma:A2} follows from 
\eqref{eq:E12}, \eqref{eq:vTv} and
\begin{displaymath}
  \E d^T d = (m-1) [ (n_E-m-1)^{-1} - (n_E
  -m)^{-1}]. 
\end{displaymath}
\end{proof}

\section{Analysis of Error Terms}
\label{sec:analysis-error-terms}
\setcounter{lemma}{0}

Subsections \ref{sec:1} through \ref{sec:4}
provide supporting details, some heuristic, for claims 
(5.15) and (5.16) about the error terms in Propositions 3 and 4.
\subsection{Study of $T_1 = 2 \epsilon \sqrt{z} E^{b1}$}
\label{sec:1}
We recall that
\begin{equation}
  \label{eq:twocases}
  z \sim
  \begin{cases}
    \chi_{n_H}^2 \\
    \omega^{-1} \chi_{n_H}^2 (\omega),
  \end{cases} 
  \qquad
  E^{b1}  = \mathring{\mathbf{b}}^T E^{-1} \be_1.
\end{equation}

\begin{proposition} \label{lem:a}
  With $\nu = n_E - m$, we have
  \begin{displaymath}
    \text{\rm var } T_1 
     = \frac{ 4 \epsilon^2 \E z \cdot (m-1)}{\nu(\nu -1)(\nu -3)}.
  \end{displaymath}
\end{proposition}
\begin{proof}
Since $\mathring{\mathbf{b}}$ is Gaussian with mean zero, $T_1$ is
symmetric and $\E T_1 = 0$. Consequently
\begin{displaymath}
  \text{var } T_1 
   = \E T_1^2 
   = 4 \epsilon^2 \E z \cdot \E (E^{b1})^2.
\end{displaymath}
To evaluate $\E (E^{b1})^2$, recall that $\mathbf{b}$ is
independent of $E$, and that $\E \mathring{\mathbf{b}} \mathring{\mathbf{b}}^T
= P_2$.
Then appeal to the formula for $\E (W^{-1} A W^{-1})$ with $W \sim 
W_p(n, \Sigma)$ given for example in 
\citet[Thm. 2.2.7(2)]{fus10}.
Indeed, with $A = P_2, \Sigma = I$ and $c_2 = 1/[\nu(\nu-1)(\nu-3)]$,
we have $\Sigma^{-1} A \Sigma^{-1} = \Sigma^{-1} A^T \Sigma^{-1} = P_2$
and so
\begin{displaymath}
  \E (E^{b1})^2
   = \E [\be_1^T E^{-1} P_2 E^{-1} \be_1 ]
   = c_2 \be_1^T [P_2 + \text{tr}(P_2) I] \be_1
   = \frac{m-1}{\nu(\nu -1)(\nu -3)}.
\end{displaymath}
The proposition follows by combining the two displays.  
\end{proof}

\subsection{Study of \, $T_2 = \epsilon^2 R$}
\label{sec:2}
\begin{proposition}
  \label{prop:T2}
  \begin{displaymath}
    \text{\rm var } T_2 
      \asymp \epsilon^4 \cdot \frac{m n_H}{n_E^4} \cdot(m + n_H).
  \end{displaymath}
\end{proposition}

From the main text, recall that $R = \mathbf{d}^T Z \mathbf{d}$ with
$\mathbf{d} = P_2 E^{-1} \be_1/ \sqrt{E^{11}}$ independently of  $Z \sim
W_{m-1}(n_H, I)$. We first express $\text{\rm var} R$ in terms of 
$\mathbf{d}^T \mathbf{d}$ by averaging over $Z$.

\begin{lemma} \label{lem:b} \qquad \qquad 
  $\text{\rm var } R = n_H^2 \text{\rm var}(\mathbf{d}^T \mathbf{d})
  + 2 n_H \E (\mathbf{d}^T \mathbf{d})^2.$
\end{lemma}
\begin{proof}
  Use the formula $\text{\rm var } R = 
\text{\rm var } [\E(R|E)] + \E \text{\rm var }(R|E)$. For $\mathbf{d}$
fixed, we have $\mathbf{d}^T Z \mathbf{d} \sim \mathbf{d}^T \mathbf{d}
\cdot \chi_{n_H}^2$ and so the result follows from
\begin{displaymath}
  \E (R|E) = n_H \mathbf{d}^T \mathbf{d}, \qquad 
  \text{var }(R|E) = 2 n_H (\mathbf{d}^T \mathbf{d})^2.  
\end{displaymath}
\end{proof}

\begin{lemma} \label{lem:c}
  Let $M \sim W_{m-1}(n_E-1,I)$ be independent of $\mathbf{w} \sim
  N_{m-1}(0,I)$. Then
  \begin{equation}
    \label{eq:smw}
    \mathbf{d}^T \mathbf{d} 
     \ \stackrel{\mathcal{D}}{\sim} \ \text{\rm tr}[M^{-1} - (M + \mathbf{w} \mathbf{w}^T)^{-1}]
     = \frac{\mathbf{w}^T M^{-2} \mathbf{w}}{1 + \mathbf{w}^T M^{-1} \mathbf{w}}.
  \end{equation}
\end{lemma}
\begin{proof}
  In the proof of Lemma \ref{lemma:A2}  it was shown that
  \begin{displaymath}
    \mathbf{d}^T \mathbf{d}
     = \text{tr}(E^{22} - E_{22}^{-1}),
  \end{displaymath}
and also that $M = (E^{22})^{-1} \sim W_{m-1}(n_E -1, I)$.
Now appealing to \citet[Cor. 3.4.6.1(b)]{mkb79}, we have
$E_{22} - M \sim W_{m-1}(1,I)$ independently of $M$.
Hence $E_{22} \stackrel{\mathcal{D}}{\sim} M + \mathbf{w} \mathbf{w}^T$ and
\begin{displaymath}
  \mathbf{d}^T \mathbf{d}
     = \text{tr} [M^{-1} - (M+ \mathbf{w} \mathbf{w}^T)^{-1}].
\end{displaymath}
From the Sherman-Morrison-Woodbury formula, the right side equals
\begin{displaymath}
  \text{tr} \Big[ \frac{M^{-1} \mathbf{w} \mathbf{w}^T M^{-1}}{1 + \mathbf{w}^T M^{-1} \mathbf{w}} \Big]
  = \frac{ \mathbf{w}^T M^{-2} \mathbf{w}}{1 + \mathbf{w}^T M^{-1} \mathbf{w}}   
\end{displaymath}
\end{proof}

We need approximations to moments of Gaussian quadratic forms in
powers of inverse Wishart matrices. A heuristic argument is given below.
\begin{lemma} \label{lem:d}
  Suppose that $M \sim W_p(n,I)$ independently of $\mathbf{z} \sim N_p(0,I)$. 
Let $ k \in \mathbb{N}$ and $n \geq 2 p$. Then
\begin{align*}
  \E [ \mathbf{z}^T M^{-k} \mathbf{z}] 
    & \asymp \frac{p}{n^k}, \qquad 
        \E [ \mathbf{z}^T M^{-k} \mathbf{z}]^2 \asymp \frac{p^2}{n^{2k}}, \qquad \text{and} \\
  \text{\rm var}[ \mathbf{z}^T M^{-k} \mathbf{z}]  
    & \asymp \frac{p}{n^{2k}}.
\end{align*}
\end{lemma}

\begin{lemma} \label{lem:e}
  For $m \leq 2 n_E$,
\begin{displaymath}
   \text{var } R 
   \asymp \frac{m n_H^2}{n_E^4} +  \frac{m^2 n_H}{n_E^4}
   \asymp \frac{m n_H}{n_E^4} \max(m, n_H).
\end{displaymath}
\end{lemma}
\begin{proof}[Heuristic argument for 
Lemma \ref{lem:e}]
Making the substitutions $m$ for $p$ and $n_E$ for $n$ in Lemma
\ref{lem:d}, we have 
 \begin{displaymath}
   \E (\mathbf{w}^T M^{-1} \mathbf{w}) \asymp m/n_E,  \qquad 
   \text{SD}(\mathbf{w}^T M^{-1} \mathbf{w}) \asymp \sqrt{m}/n_E,
 \end{displaymath}
and so for $m \leq 2 n_E$, say, we may ignore the denominator in
\eqref{eq:smw}. 
Hence--and again using Lemma \ref{lem:d}--we may approximate the terms
in Lemma \ref{lem:b} as follows:
\begin{displaymath}
  \text{\rm var}(\mathbf{d}^T \mathbf{d}) 
  \approx \text{\rm var}[ \mathbf{w}^T M^{-2} \mathbf{w}] 
      \asymp \frac{m}{n_E^4}, \qquad 
  \E (\mathbf{d}^T \mathbf{d})^2
  \approx \E [ \mathbf{w}^T M^{-2} \mathbf{w}]^2 
      \asymp \frac{m^2}{n_E^4}.
\end{displaymath}
Hence from Lemma \ref{lem:b}, we find that, as claimed
\begin{displaymath}
  \text{var } R 
   \asymp \frac{n_H^2 m}{n_E^4} +  \frac{n_H m^2}{n_E^4}.  
\end{displaymath}
\end{proof}

\subsection{Conclusions about error terms}
\label{sec:3}
We now have the tools to argue
the Claim in \eqref{eq:Var-Ti1}, that if $m, n_H \leq
2 n_E$, then  
\begin{displaymath}
  \max_i \text{var} \, (\epsilon^{-2} T_i)
    \leq \frac{c}{n_E} \frac{m}{n_E} \frac{n_H}{n_E}
    \begin{cases}
      1 + \lambda_H  &  \text{Case 3} \\
      1 + \omega/n_H  &  \text{Case 4}.
    \end{cases}
\end{displaymath}
For Case 3, we have from Propositions \ref{lem:a} and \ref{prop:T2}
that
\begin{displaymath}
  \text{\rm var} ( \epsilon^{-2} T_1)
  \asymp \frac{\epsilon^{-2} m n_H}{n_E^3}
  = \frac{m}{n_E} \frac{n_H}{n_E} \frac{1+\lambda_H}{n_E}, \qquad
  \text{\rm var} ( \epsilon^{-2} T_2)
  \asymp \frac{m}{n_E} \frac{n_H}{n_E} \frac{m+n_H}{n_E^2}
  \leq \frac{c}{n_E} \frac{m}{n_E}  \frac{n_H}{n_E},
\end{displaymath}
where we use the assumptions that $m, n_H \leq 2 n_E$. 

For Case 4,
\begin{displaymath}
  \text{\rm var} ( \epsilon^{-2} T_1)
  \asymp \frac{\epsilon^{-2} m \, \E z}{n_E^3}
  \asymp \frac{m}{n_E^2} \frac{\omega + n_H}{n_E}, \qquad 
  \text{\rm var} ( \epsilon^{-2} T_2)
  \asymp \frac{m}{n_E^2} \frac{n_H}{n_E} \frac{m+n_H}{n_E}
\end{displaymath}
Consequently, 
\begin{align*}
  \max_i \text{var} \, (\epsilon^{-2} T_i)
  & \leq c \frac{m}{n_E^2} \max \Big( \frac{\omega + n_H}{n_E}, 
          \frac{n_H}{n_E} \frac{m+n_H}{n_E} \Big) \\
  & = \frac{c}{n_E} \frac{m}{n_E} \frac{n_H}{n_E}
         \max \Big( n_H^{-1} \omega + 1, \frac{m+n_H}{n_E} \Big) \\
  & \leq \frac{c}{n_E} \frac{m}{n_E} \frac{n_H}{n_E} (1 + \omega/n_H).
\end{align*}

\subsection{Heuristic argument for Claim \ref{lem:d}}
\label{sec:4}
  We have
\begin{displaymath}
    \E [\mathbf{z}^T M^{-k} \mathbf{z}]
   = \E \E[ \text{tr }M^{-k}\mathbf{z} \mathbf{z}^T | M] 
   = \E \, \text{tr } M^{-k}
   = \E \sum_1^p \lambda_i^{-k}.
\end{displaymath}
According to the Mar\u{c}enko-Pastur law, the empirical
distribution of the eigenvalues $\{\lambda_i \}$ of $M$ converges to a
law supported in
\begin{displaymath}
  [(\sqrt n - \sqrt p)^2, (\sqrt n + \sqrt p)^2] 
   \subset n [a, b],
\end{displaymath}
where if $2p \leq n$, the constants $a = (1 - 2^{-1/2})^2, 
b = (1 + 2^{-1/2})^2$. Hence the first claim follows from
\begin{equation}
  \label{eq:moms}
  \E \sum_1^p \lambda_i^{-k} \asymp p n^{-k}.
\end{equation}

For the second claim, write
\begin{displaymath}
  \E [ \mathbf{z}^T M^{-k} \mathbf{z}]^2
  = \E \text{tr}( M^{-k} \mathbf{z} \mathbf{z}^T M^{-k} \mathbf{z} \mathbf{z}^T ).
\end{displaymath}
We use \citet[Thm. 2.2.6(3), p. 35]{fus10} with $A = B = M^{-k}$, and
$\Sigma = I, n = 1$ to write this as 
\begin{displaymath}
  \E [ 2 \text{tr } M^{-2k} + (\text{tr } M^{-k})^2 ]
  \asymp p n^{-2k} + p^2 n^{-2k} 
  \asymp p^2 n^{-2k},
\end{displaymath}
by arguing as in \eqref{eq:moms}.

Turning to the variance term, we use the decomposition
\begin{displaymath}
  \text{var}(\mathbf{z}^T M^{-k} \mathbf{z}) 
    = \text{var } \E (\mathbf{z}^T M^{-k} \mathbf{z} |M) 
      + \E \, \text{var}(\mathbf{z}^T M^{-k} \mathbf{z} |M).
\end{displaymath}
Condition on $M$ and use its spectral decomposition $M = U \Lambda
U^T$ for $U$ a $p \times p$ orthogonal matrix. Then
\begin{displaymath}
  \mathbf{z}^T M^{-k} \mathbf{z} = (U^T \mathbf{z})^T \Lambda^{-k} U^T
  \mathbf{z}   = \sum_1^p \lambda_i^{-k} x_i^2,
\end{displaymath}
for $\mathbf{x} = U^T \mathbf{z} \sim N_p (0,I)$. Hence, since 
$\{ x_i^2 \}$ are i.i.d. $\chi_1^2$,
\begin{displaymath}
  \text{var}(\mathbf{z}^T M^{-k} \mathbf{z} | M)
   = \sum_1^p 2 \lambda_i^{-2k}
   = 2 \text{tr } M^{-2k},
\end{displaymath}
and so
\begin{displaymath}
  \text{var}(\mathbf{z}^T M^{-k} \mathbf{z})
    = \text{var}[ \text{tr}(M^{-k})] + 2 \E \, \text{tr }M^{-2k}.
\end{displaymath}
From the argument at \eqref{eq:moms}, $\E \, \text{tr }M^{-2k} \asymp
p n^{-2k}$. 
We further claim that $\text{var}[ \text{tr}(M^{-k})]$ is of smaller
order, and in particular
\begin{displaymath}
  \text{var}[ \text{tr}(M^{-k})] \leq C n^{-2k}.
\end{displaymath}
Here the argument becomes more heuristic: write
\begin{displaymath}
  \text{tr}(M^{-k}) = \sum_1^p \lambda_i^{-k} = n^{-k} \sum_1^p \mu_i^{-k},
\end{displaymath}
where $\{ \mu_i \}$ are eigenvalues of a normalized Wishart matrix
$M/n$ with the limiting Mar\u{c}enko-Pastur distribution supported on 
$[a(c), b(c)]$ for $c = \lim p/n$. 

Now $S = \sum_1^p \mu_i^{-k}$ is a linear eigenvalue
statistic. If $p /n \to c \in (0, \infty)$, then from \citet{Bai2004}, 
$\text{var } S$ is $O(1)$.
Heuristically, one expects this to be true also for $p = o(n)$, so
that
\begin{displaymath}
  \text{var}(\text{tr } M^{-k})
    = n^{-2k} \text{var} \Bigl( \sum_1^p \mu_i^{-k} \Bigr) 
    \leq C n^{-2k}.
\end{displaymath}

\textit{Remark.} An explicit calculation of $\text{var}(\text{tr }
W^{-1})$---the case $k=1$ above---is possible for $W \sim W_p(n, I)$. 
With $\nu = n -p, c_1 = (\nu-2) c_2$ and $c_2 =
1/[\nu(\nu-1)(\nu-3)]$, we have from \citet[p. 35, 36]{fus10}
\begin{align*}
  \E \, \text{tr } W^{-1} 
    & = p (\nu-1)^{-1} \\
  \E \, (\text{tr } W^{-1})^2 
    & = c_1 p^2 + 2 c_2 p 
      = p c_2 [ p(\nu-2) + 2].
\end{align*}
Consequently
\begin{align*}
  \text{var}(\text{tr }W^{-1})
   & = p^2 c_2(\nu -2) + 2 p c_2 - p^2 (\nu-1)^{-2} \\
   & = \frac{p^2}{\nu (\nu-1)^2 (\nu - 3)} \{ 2 + 2 p^{-1}(\nu -1) \} 
    \asymp \frac{2 p \max (p, \nu) }{\nu^4} 
     \leq C p / n^3.
\end{align*}

\section{Remark on perturbation expansions}
\label{sec:remark-pert-expans}

The arguments given for Propositions \ref{prop:h1} and \ref{prop:h1_M}
deliberately skirted a technical point which is addressed briefly
here.
In the stochastic version of model \eqref{eq:xidef}---augmented with
\eqref{eq:dist1}---the $u_i$ also depend on the small parameter
$\epsilon$. 
To clarify this, let us introduce a parameter $\epsilon_2$ and
explicit independent $N(0,1)$ variates $w_i$ so that the $u_i$ in
\eqref{eq:dist1} may be written
\begin{displaymath}
  u_i =
  \begin{cases}
    \sqrt{\epsilon_2^2 + \lambda_H} w_i  &  \text{Case 1} \\
    \mu_i + \epsilon_2 w_i               &  \text{Case 2}
  \end{cases}
\end{displaymath}
We think of the observations \eqref{eq:xidef} as having the form
\begin{displaymath}
      \mathbf{x}_i = u_i(\epsilon_2) \mathbf{e}_1 + \epsilon_1 
        \mathring{\bm{\xi}}_i
\end{displaymath}
and write the largest eigenvalue of $H = \sum_1^n \mathbf{x}_i \mathbf{x}_i^T$
as $\ell_1(\epsilon_1, \epsilon_2, \varsigma)$, where
$\varsigma = \{ (\bm{\xi}_i, w_i), i = 1, \ldots, n \}$ indicates the
random variables.

Lemma \ref{L:perturbation} provides an approximation holding
$\epsilon_2$ and $\varsigma$ fixed:
\begin{displaymath}
  \ell_1(\epsilon_1, \epsilon_2, \varsigma)
  = \ell_1^\circ(\epsilon_1, \epsilon_2, \varsigma) 
    + r(\epsilon_1, \epsilon_2, \varsigma),
\end{displaymath}
in which, for some $\epsilon_1(\epsilon_2, \varsigma) \in (0, \epsilon_2)$,
\begin{displaymath}
  \ell_1^\circ
    = \sum_{k=0}^2 \frac{\epsilon_1^{2k}}{(2k)!} D_1^{2k} \ell_1(0,
    \epsilon_2, \varsigma), \qquad
  r = \frac{\epsilon_1^6}{6!} D_1^6 \ell_1(\epsilon_1(\epsilon_2,
  \varsigma), \epsilon_2, \varsigma),
\end{displaymath}
where $D_1$ denotes the derivative of $\ell_1$ w.r.t the first
co-ordinate. 
For Propositions \ref{prop:h1} and \ref{prop:h1_M}, we equate
$\epsilon_1$ and $\epsilon_2$ and assert that 
\begin{displaymath}
    \ell_1(\epsilon, \epsilon, \varsigma)
  = \ell_1^\circ(\epsilon, \epsilon, \varsigma) + o_p(\epsilon^4).
\end{displaymath}
The latter statement may be justified as follows,
in view of the structure of simple rank one structure of 
$A_0$ and the Gaussian distribution of
$\varsigma$. 
Given $\eta, \delta > 0$, we can find $\epsilon_0$ small, $M$ large
and an event $A_\delta$ of probability at least $1 - \delta$ such that for
$|\epsilon_i| \leq \epsilon_0$,
\begin{displaymath}
  \sup_{\varsigma \in A_\delta} |D_1^6 \ell_1(\epsilon_1, \epsilon_2,
  \varsigma)| < M.
\end{displaymath}
Consequently, for $\epsilon < \epsilon_0$ such that $\epsilon^2 M/6! <
\eta$,
\begin{displaymath}
  P \{ \epsilon^{-4} |\ell_1(\epsilon, \epsilon, \varsigma)
  - \ell_1^\circ(\epsilon, \epsilon, \varsigma)| > \eta \} 
    < \delta.
\end{displaymath}

A remark of the same general nature would also apply to the proof of
Proposition 
\ref{prop:MANOVA}, regarding the parameter $\omega$ in
\eqref{eq:z-dist}.

\end{document}